\documentclass[12pt,reqno]{amsart}
\usepackage{latexsym,amsmath,amsfonts,amscd,amssymb}
\usepackage{graphics}

\usepackage{tikz-cd}

\newcommand{\inc}{\hookrightarrow}
\newcommand{\la}{\langle}
\newcommand{\ra}{\rangle}
\newcommand{\x}{\times}
\newcommand{\ox}{\otimes}
\newcommand{\too}{\longrightarrow}
\newcommand{\bd}{\partial}

\newcommand{\SC}{{\mathcal{C}}}
\newcommand{\SD}{{\mathcal{D}}}

\newcommand{\ZZ}{\mathbb{Z}}
\newcommand{\CC}{\mathbb{C}}
\newcommand{\CP}{\mathbb{C}P}
\newcommand{\RR}{\mathbb{R}}
\newcommand{\TT}{\mathbb{T}}
\newcommand{\QQ}{\mathbb{Q}}
\newcommand{\id}{\operatorname{id}}

\newcommand{\U}{\operatorname{U}}
\newcommand{\SO}{\operatorname{SO}}
\newcommand{\GL}{\operatorname{GL}}

\DeclareMathOperator{\lcm}{lcm}

\DeclareMathOperator{\Id}{Id}

\DeclareMathOperator{\End}{End}

\newcommand{\cO}{\mathcal O}
\newcommand{\PP}{\mathbb P}

\newtheorem{theorem}{Theorem}
\newtheorem{proposition}[theorem]{Proposition}
\newtheorem{lemma}[theorem]{Lemma}
\newtheorem{definition}[theorem]{Definition}

\newtheorem{corollary}[theorem]{Corollary}
\newtheorem{remark}[theorem]{Remark}

\title[Homology Smale-Barden which are K-contact but not Sasakian]{Homology Smale-Barden manifolds 
with K-contact but not Sasakian structures}

\author[V. Mu\~{n}oz]{Vicente Mu\~{n}oz}
\address{Departamento de Algebra, Geometr\'{\i}a y Topolog\'{\i}a, Universidad de M\'alaga,
Campus de Teatinos, s/n, 29071 Malaga, Spain}
\email{vicente.munoz@uma.es}

\author[J.A. Rojo]{Juan Angel  Rojo}
\address{Facultad de Ciencias Matem\'aticas, Universidad
Complutense de Madrid, Plaza de Ciencias 3, 28040 Madrid, Spain}
\address{Instituto de Ciencias Matem\'aticas (CSIC-UAM-UC3M-UCM),
C/ Nicol\'as Cabrera 15, 28049 Madrid, Spain}
\email{jarcarulli@ucm.es}

\author[A. Tralle]{Aleksy Tralle}
\address{Department of Mathematics and Computer Science, University of Warmia
and Mazury, S\l\/oneczna 54, 10-710, Olsztyn, Poland}
\email{tralle@matman.uwm.edu.pl}



\begin{document}

\begin{abstract}
Koll\'ar has found  subtle obstructions to the existence of Sasakian structures on $5$-dimensional manifolds. 
In the present article we develop methods of using these obstructions to distinguish K-contact manifolds from 
Sasakian ones. In particular, we find the first example of a closed $5$-manifold $M$ with $H_1(M,\ZZ)=0$ which is 
K-contact but which carries no semi-regular Sasakian structures. 
\end{abstract}

\maketitle

\section{Introduction}\label{sec:1}

Sasakian geometry has become an important and active subject, especially after the appearance of
the fundamental treatise of Boyer and Galicki \cite{BG}. Chapter 7 of this
book contains an extended discussion of the topological problems in the theory of Sasakian, and, more generally,
K-contact manifolds. These are odd-dimensional analogues to K\"ahler and symplectic manifolds, respectively. 

The precise definition is as follows.
Let $(M,\eta)$ be a co-oriented contact manifold with a contact form
$\eta\in \Omega^1(M)$, that is $\eta\wedge (d\eta)^n>0$ everywhere, with $\dim M=2n+1$. 
We say that $(M,\eta)$ is {\em K-contact} if there is an
endomorphism $\Phi$ of $TM$ such that:
 \begin{itemize}
\item $\Phi^2=-\Id + \xi\otimes\eta$, where $\xi$ is the Reeb
vector field of $\eta$ (that is $i_\xi \eta=1$, $i_\xi (d\eta)=0$),
\item the contact form $\eta$ is compatible with $\Phi$ in the sense that
$d\eta (\Phi X,\Phi Y)\,=\,d\eta (X,Y)$, for all vector fields $X,Y$,
\item $d\eta (\Phi X,X)>0$ for all nonzero $X\in \ker \eta$, and
\item the Reeb field $\xi$ is Killing with respect to the
Riemannian metric defined by the formula
 $g(X,Y)\,=\,d\eta (\Phi X,Y)+\eta (X)\eta(Y)$.
\end{itemize}
In other words, the endomorphism $\Phi$ defines a
complex structure on $\SD=\ker \eta$ compatible with $d\eta$, hence
$\Phi$ is orthogonal with respect to the metric
$g|_\SD$. By definition, the Reeb vector
field $\xi$ is orthogonal to $\ker \eta$, and it is a 
Killing vector field.

Let $(M,\eta,g,\Phi)$ be a K-contact manifold. Consider the contact cone as the Riemannian manifold
$C(M)=(M\times\mathbb{R}^{>0},t^2g+dt^2)$.
One defines the almost complex structure $I$ on $C(M)$ by:
 \begin{itemize}
\item $I(X)=\Phi(X)$ on $\ker\eta$,
\item $I(\xi)=t{\partial\over\partial t},\,I(t{\partial\over\partial t})=-\xi$, for the Killing vector field $\xi$ of $\eta$.
\end{itemize}
We say that $(M,\eta,\Phi,g,I)$ is  {\it Sasakian} if $I$ is integrable.
Thus, by definition, any Sasakian manifold is K-contact.

There is much interest on constructing K-contact manifolds which do not admit Sasakian structures.
The odd Betti numbers up to dimension $n$ of Sasakian $2n+1$-manifolds must be even. 
The parity of $b_1$ was used to produce the first examples of K-contact manifolds with no Sasakian structure
\cite[example 7.4.16]{BG}. More refined tools are needed in the case of even Betti numbers. The 
cohomology algebra of a Sasakian manifold satisfies a hard Lefschetz property \cite{CNY}. Using
it examples of K-contact non-Sasakian manifolds are produced in \cite{CNMY} in dimensions $5$ and $7$. These
examples are nilmanifolds with even Betti numbers, so in particular they are not simply connected.

The fundamental group can also be used to construct K-contact non-Sasakian manifolds.
Fundamental groups of Sasakian manifolds are called Sasaki groups, and satisfy
strong restrictions. Using this it is possible to construct (non-simply connected) compact 
manifolds which are K-contact but not Sasakian \cite{Chen}. Also it has been used to
provide an example of a solvmanifold of dimension $5$ which satisfies the hard Lefschetz 
property and which is K-contact and not Sasakian \cite{CNMY2}.

When one moves to the case of simply connected manifolds, K-contact non-Sasakian 
examples of any dimension $\geq 9$ were constructed in \cite{HT} using the evenness of the third Betti
number of a compact Sasakian manifold. Alternatively, using the hard Lefschetz property
for Sasakian manifolds there are examples \cite{Lin} of simply connected K-contact non-Sasakian
manifolds of any dimension $\geq 9$.

In \cite{Tievsky} and in \cite{BFMT} the rational homotopy type of Sasakian manifolds is studied.
In \cite{BFMT} it is proved that all higher order Massey products for simply connected Sasakian manifolds vanish,
although there are Sasakian manifolds with non-vanishing triple Massey products.
This yields examples of simply connected K-contact non-Sasakian manifolds in dimensions $\geq 17$. 
However, Massey products are not suitable for the analysis of lower dimensional manifolds.

The problem of the existence of simply connected K-contact non-Sasakian compact manifolds 
(open problem 7.4.1 in \cite{BG}) is still open in dimension $5$. It was solved for dimensions $\geq 9$  
in \cite{CNY,CNMY,HT} and for dimension $7$ in \cite{MT} by a combination of various techniques 
based on the homotopy theory and symplectic geometry. In the least possible dimension the problem appears to 
be much more difficult. Here one
has to use the arguments of \cite{Kollar} which give subtle obstructions associated to the
classification of K\"ahler surfaces.
By definition, a simply connected compact oriented $5$-manifold is called a {\it Smale-Barden manifold}. 
These manifolds are classified topologically by $H_2(M,\mathbb{Z})$ and the second Stiefel-Whitney class. 
Chapter 10 of the book by Boyer and Galicki is devoted to a description of some Smale-Barden manifolds which 
carry Sasakian structures. The following problem is still open (open problem 10.2.1 in \cite{BG}). 

{\it Do there exist Smale-Barden manifolds which carry K-contact but do not carry Sasakian structures?}

In \cite{Kim} it is claimed a solution of this question, but unfortunately the main result of that paper is flawed,
as we explain in Remark \ref{rem:Kim}.

In the present paper we make the first step towards a positive answer for the above question.
A {\it homology Smale-Barden manifold} is a compact $5$-dimensional manifold with $H_1(M,\mathbb{Z})=0$. 
A Sasakian structure is regular if the leaves of the Reeb flow are a folation by circles with the
structure of a circle bundle over a smooth manifold. The Sasakian structure is quasi-regular if the foliation
is a Seifert circle bundle over a (cyclic) orbifold. It is semi-regular if this foliation has only locus of non-trivial
isotropy of codimension $2$, that is, if the base orbifold is a topological manifold. Any manifold
admitting a Sasakian structure has also a quasi-regular Sasakian structure. Semi-regularity is 
only a small extra requirement. With this notions, our main result is:

\begin{theorem}\label{thm:main} 
There exists a homology Smale-Barden manifold which admits a semi-regular K-contact structure but which 
does not carry any semi-regular Sasakian structure.
\end{theorem}

In order to put our result into a general context, it is worth recalling Koll\'ar's obstructions 
to Sasakian structures \cite{Kollar}. If a $5$-dimensional manifold $M$  has a Sasakian structure, then it has
a quasi-regular Sasakian structure. Then it is a Seifert bundle structure over a K\"ahler orbifold
$X$ with isotropy locus a collection of complex curves. If $H_1(M,\ZZ)=0$, then these curves span $H_2(X,\QQ)$. 
The second homology $H_2(M,\ZZ)$ allows to 
recover the genus of these curves,  if they are disjoint. In \cite{Kollar}, $5$-manifolds $M$ are constructed
which are Seifert bundles over $4$-orbifolds $X$ with isotropy being surfaces not satisfying the adjunction equality,
hence $X$ cannot be K\"ahler and $M$ cannot be Sasakian. 

To produce K-contact $5$-dimensional manifolds we need to produce symplectic $4$-dimensional orbifolds with 
suitable symplectic surfaces spanning the second homology. Such K-contact $5$-manifold cannot admit a Sasakian structure if we prove
that such configuration of surfaces (genus, disjointness condition and spanning the second homology) cannot
be produced for a K\"ahler orbifold with complex curves. We propose the following conjecture:

{\it 
  There does not exist a K\"ahler manifold or a K\"ahler orbifold $X$ with $b_1=0$ and with
  disjoint complex curves spanning $H_2(X,\QQ)$, all of genus $g\geq 1$.
}

We give the first result in this direction (Theorem \ref{thm:4g+5}). Our construction of a K-contact $5$-manifold
which does not admit a Sasakian structure relies on producing a symplectic $4$-manifold with disjoint
symplectic surfaces spanning the second homology,
of genus $g\geq 1$, but also with genus $g\leq 3$, to fit with our needs in Theorem \ref{thm:4g+5}.
This is the content of the delicate construction in Section \ref{sec:5}.

Finally, our examples are K-contact manifolds which are quasi-regular 
but do not admit a regular K-contact structure (Remark \ref{rem:suggestion}). There are 
examples of Sasakian quasi-regular manifolds in \cite{BGe} which do not admit
regular structures. These are spin Smale-Barden manifolds with $H_2(X,\ZZ)$ torsion and 
non-zero, and by Remark \ref{rem:e2}, they can not admit regular Sasakian (or K-contact)
structures.


\medskip

\noindent\textbf{Acknowledgements.} We are grateful Jaume Amor\'os and Charles Boyer for useful conversations.
We also thank Dominic Joyce, Simon Donaldson and Ivan Smith for their comments.
Thanks also to the referee for a careful reading and correcting an erroneous statement.
The first and second authors were partially supported by 
Project MINECO (Spain) MTM2015-63612-P.
The second author acknowledges financial support by the International PhD program La Caixa-Severo Ochoa.

\section{$4$-dimensional cyclic orbifolds}\label{sec:2}

An $n$-dimensional (differentiable) orbifold is a space endowed with an atlas $\{(\tilde U_{\alpha},\phi_{\alpha},\Gamma_{\alpha})\}$, 
where $\tilde U_\alpha\subset\RR^n$, $\Gamma_{\alpha} < \GL(\RR^n)$ is a finite group acting linearly, and 
$\phi_{\alpha}:\tilde U_{\alpha} \to U_{\alpha} \subset X$ is a $\Gamma_{\alpha}$-invariant map which induces a 
homeomorphism $\tilde U_{\alpha}/\Gamma_\alpha \cong U_{\alpha}$ onto an open set $U_\alpha$ of $X$. 
There is also a condition of compatibility 
of charts: for each point $p \in U_{\alpha} \cap U_{\beta}$ there is some $U_\gamma\subset 
U_{\alpha} \cap U_{\beta}$ with $p \in U_\gamma$, monomorphisms
$\imath_{\gamma\alpha}: \Gamma_\gamma\inc \Gamma_\alpha$,
$\imath_{\gamma\beta}: \Gamma_\gamma\inc \Gamma_\beta$, and 
open embeddings $f_{\gamma\alpha}:\tilde U_\gamma \to \tilde U_\alpha$, 
$f_{\gamma\beta}:\tilde U_\gamma \to \tilde U_\beta$, which satisfy $\imath_{\gamma\alpha}(g)(f_{\gamma\alpha}(x))=
f_{\gamma\alpha}(g(x))$ and $\imath_{\gamma\beta}(g)(f_{\gamma\beta}(x))=
f_{\gamma\beta}(g(x))$, for $g\in \Gamma_\gamma$. 

As the groups $\Gamma_\alpha$ are finite, we can arrange (after a suitable conjugation) that $\Gamma_\alpha <
\operatorname{O}(n)$. The orbifold is {\em orientable} if all $\Gamma_\alpha < \SO(n)$ and the embedings $f_{\gamma\alpha}$ preserve orientation. 
Note that for any point $x\in X$, we can arrange always a chart $\phi:\tilde U\to U$ with $\tilde U\subset \RR^n$ is a 
ball centered at $0$ and $\phi(0)=x$, and $\tilde U/\Gamma\cong U$, with $\Gamma<\SO(n)$. In this case, we call $\Gamma$
the {\em isotropy group} at $x$. A {\em cyclic} orbifold has all isotropy groups which are cyclic groups $\Gamma\cong \ZZ_m$,
and $m=m(x)$ is the order of the isotropy at $x$. In this paper, all we shall work exclusively with $4$-dimensional cyclic oriented
orbifolds, which we shall address just as orbifolds. 

Let $X$ be such an orbifold. Take $x\in X$ and a chart $\phi:\tilde U\to U$ around $x$. Let 
$\Gamma=\mathbb{Z}_m<\SO(4)$ be the isotropy group. Then $U$ is homeomorhic to an open neighbourhood of 
$0\in \mathbb{R}^4/\mathbb{Z}_m$. A matrix of finite order in $\SO(4)$ is conjugate to a diagonal matrix in $\U(2)$ of
the type $(\exp(2\pi i j_1/m),\exp(2\pi i j_2/m))=(\xi^{j_1},\xi^{j_2})$, where $\xi=e^{2\pi i/m}$. Therefore we can
suppose that $\tilde U\subset \CC^2$ and $\Gamma=\ZZ_m=\la \xi\ra\subset \U(2)$ acts on $\tilde U$ as
 \begin{equation} \label{action}
   \xi \cdot (z_1,z_2) := (\xi^{j_1} z_1,\xi^{j_2}z_2).
 \end{equation}
Here $j_1,j_2$ are defined modulo $m$.
As the action is effective, we have $\gcd(j_1,j_2,m)=1$. Let us list the possible local models for an action given by the
formula (\ref{action}). 

We call $x\in X$ a regular point if $m(x)=1$, otherwise we call it a (non-trivial) isotropy point. We say that $D\subset X$ is an
isotropy surface of multiplicity $m$ if $D$ is closed, and there is a dense open subset $D^\circ\subset D$ which is a surface
and $m(x)=m$, for $x\in D^\circ$.
From the topological point of view, we call $x\in X$ a smooth point if a neighbourhood of $x$ is 
homeomorphic to a ball in $\RR^4$, and singular otherwise. Clearly a regular point is smooth, but not conversely as we
shall see next.

\begin{proposition} \label{prop:models}
Let $X$ be a (cyclic, oriented, $4$-dimensional) orbifold and $x\in X$ with local model $\CC^2/\ZZ_m$. Then there
are at most two isotropy surfaces $D_i$, with multiplicity $m_i | m$, through $x$. If there are two such
surfaces $D_i,D_j$, then they intersect transversely and $\gcd(m_i,m_j)=1$. The fundamental group of the link of
$x$ has order $d$ with $(\prod m_i)  d=m$, the product over all $m_i$ such that $x\in D_i$. 
So the point is smooth if and only if\/ $\prod m_i=m$.
\end{proposition}

\begin{proof}
For an action given by (\ref{action}), we 
set $m_1:=\gcd(j_1,m)$, $m_2:=\gcd(j_2,m)$. Note that $\gcd(m_1,m_2)=1$, so we can
write $m_1m_2 d=m$, for some integer $d$. Put $j_1=m_1 e_1$, $j_2=m_2 e_2$, $m=m_1 c_1=m_2 c_2$.
Clearly $c_1=m_2d$ and $c_2=m_1d$ and $d=\gcd(c_1,c_2)$.

We have five cases:

\begin{enumerate}
\item[(a)] $x$ is an isolated singular point. This corresponds to $m_1=m_2=1$. As $\gcd(j_1,m)=\gcd(j_2,m)=1$, 
the only fixed point is $(0,0)$ since any power of $\xi$ rotates both copies of $\mathbb{C}$ non trivially. 
In this case the quotient space is singular, and the singularity is a cone over a lens space ${S}^3/\mathbb{Z}_m$,
which is the link of the origin. Note that $d=m$.

\item[(b)] Two isotropy surfaces and $x$ is a smooth point, $m_1,m_2>1$, $d=1$. 
Let us see that the action is equivalent to the product of one action on each factor $\mathbb{C}$. 
In this case $c_2=m_1$ and $c_1=m_2$. So $\gcd (c_1,c_2)=1$ and $m=c_1c_2$. The action is given by
$\xi \cdot (z_1,z_2) := (\exp(2\pi i e_1/c_1)z_1,\exp(2\pi i e_2/c_2)z_2)$. We see that 
 $$
 \begin{array}{ccc}
 \xi^{c_1}\cdot (z_1,z_2)=(z_1, \exp(2 \pi i c_1 e_2/c_2) z_2),  \\
 \xi^{c_2}\cdot (z_1,z_2)=(\exp(2 \pi i c_2 e_1/c_1) z_1, z_2),
 \end{array} 
 $$
so the surfaces $D_1=\{(z_1,0)\}$ and $D_2=\{(0,z_2)\}$ have isotropy groups $\la \xi^{c_1}\ra=\mathbb{Z}_{m_1}$ 
and $\la\xi^{c_2}\ra=\mathbb{Z}_{m_2}$, respectively. In this case $m=m_1m_2$, $d=1$.

Note that $\ZZ_m=\la \xi^{c_1}\ra \times \la\xi^{c_2}\ra$ if and only if $d=\gcd(c_1,c_2)=1$. 
In this case the action of $\mathbb{Z}_m$ decomposes as the product of the actions of 
$\mathbb{Z}_{m_2}$ and $\mathbb{Z}_{m_1}$ on each of the factors $\CC$. The quotient space is
$\CC^2/\ZZ_m\cong \CC/\ZZ_{m_2} \x \CC/\ZZ_{m_1}$, which is homeomorphic to $\CC\x\CC$, and hence $x$ is a
smooth point (its link is $S^3$).

\item[(c)] Two isotropy surfaces intersect at $x$ and $x$ is a singular point. In this case $d=\gcd(c_1,c_2)>1$ and 
$m_1,m_2>1$. Now $\la\xi^{c_1},\xi^{c_2}\ra= \la\xi^{d}\ra= \mathbb{Z}_{m'}$ with $d\,m'=m$. 
As $m'=m_1m_2$, case (b) applies to the action of $\xi^d$ and the quotient space is
$\CC^2/\ZZ_{m'}\cong \CC/\ZZ_{m_2} \times \CC/\ZZ_{m_1}$, which is homemorphic to a ball in 
$\CC^2$ via the map $(z_1,z_2)\mapsto (w_1,w_2)=(z_1^{m_2},z_2^{m_1})$. The points of 
$D_1=\{(w_1,0)\}$ and $D_2=\{(0,w_2)\}$ define two surfaces intersecting transversely, and with multiplicities
$m_1,m_2$, respectively. 
 
Now $\xi$ acts on $\CC^2/\ZZ_{m'}$ by the formula $\xi \cdot(w_1,w_2)=(\xi^{m_2j_1} w_1,\xi^{m_1j_2}w_2)=
(\exp(2\pi ie_1/d) w_1, \exp(2\pi ie_2/d) w_2)$, where $\gcd(e_1,d)=\gcd(e_2,d)=1$. Therefore this
action falls into case (a). The quotient is therefore $\CC^2/\la \xi\ra \cong ( \CC/\ZZ_{m_2} \times \CC/\ZZ_{m_1})/\ZZ_d$,
the point $x$ has as link a lens space ${S}^3/\mathbb{Z}_d$, and the images of $D_1$ and $D_2$ are the points
with non-trivial isotropy, with multiplicities $m_1,m_2$, respectively.

\item[(d)]  One isotropy surface and $x$ is a smooth point. In this case $m_2=1$ and $m_1=m$. As $d=1$, this is basically
as case (b). The action is $\xi \cdot (z_1,z_2)=(z_1,\exp(2\pi j_2/m)  z_2)$. There is only one surface 
$D_1=\{(z_1,0)\}$ with non-trivial isotropy $m$, and all its points have the same 
isotropy. The quotient $\CC^2/\ZZ_m =\CC \x (\CC/\ZZ_m)$ is topologically smooth.

\item[(e)]  One isotropy surface and $x$ is a singular point. In this case $m_2=1$, $m_1d=m$ and $d>1$. This is 
basically as case (c). Now $c_2=m$ and $c_1=d$. Let $dm'=m$ so $m'=m_1$. The quotient space 
$\CC^2/\ZZ_{m'}\cong \CC \times \CC/\ZZ_{m_1}$ is homemorphic to a ball in 
$\CC^2$ and the points of $D_1=\{(z_1,0)\}$ define a surface with isotropy $m_1$.
Now for the quotient $\CC^2/\ZZ_m=(\CC \times (\CC/\ZZ_{m_1}))/\ZZ_d$, the image of $D_1$ consists of
points with isotropy $m_1$, except for the origin which has isotropy $m=m_1d$. The link around $x$ is
the lens space $S^3/\ZZ_d$, hence it is singular. The rest of the points of $D_1$ are smooth.

\end{enumerate}
\end{proof}

\begin{definition} \label{def:orb-smooth}
 We say that an orbifold $X$ is smooth if all its points are smooth. That is, all points of $X$
fall into cases (b) or (d) in Proposition \ref{prop:models}. This is equivalent to $X$ being a topological manifold.
\end{definition}

Note that in case (b), we can change the generator $\xi=e^{2\pi i/m}$ of $\ZZ_{m}$ to 
$\xi'=\xi^k$ for $k$ such that $k e_i \equiv 1 \pmod{m_i}$, $i=1,2$,
so that $\xi' \cdot (z_1,z_2)=(\exp(\frac{2 \pi i}{m_2})z_1, \exp(\frac{2 \pi i}{m_1})z_2)$. 
With this new generator, the action has model $\CC^2$ with the action 
$\xi\cdot (z_1,z_2)=(\xi^{m_1}z_1, \xi^{m_2} z_2)$, $\xi=e^{2\pi i/m}$. Similar remark applies to case (d).

Note that according to Definition  \ref{def:orb-smooth}, a smooth orbifold is not a smooth manifold. However,
there is a mechanism to produce a smooth orbifold from a smooth manifold. This is a standard
result, but we include the proof since we have not found it in the literature.

\begin{proposition} \label{prop:smooth->orb}
 Let $X$ be a smooth (oriented) $4$-manifold with embedded surfaces $D_i$ intersecting transversely, and 
coefficients $m_i>1$ such that
$\gcd(m_i,m_j)=1$ if $D_i$, $D_j$ intersect, then there is a smooth orbifold $X$ with isotropy surfaces $D_i$ of
multiplicities $m_i$.
\end{proposition}

\begin{proof}
We consider $X$ with its atlas as smooth manifold. We start by fixing a Riemannian metric such that in a neighbourhood of 
the (finitely many) points which are in the intersection of two of the $D_i$'s, it is standard, that is, 
for $x\in D_i\cap D_j$ there is a chart $f:B_{\epsilon}(0)\x B_{\epsilon}(0) \subset \RR^2\x\RR^2 \to U$, with $f(0,0)=x$,
$D_i\cap U=f(B_{\epsilon}(0) \x \{0\})$, $D_j\cap U=f(\{0\}\x B_{\epsilon}(0))$, 
and $g$ is the standard metric on $U$.

Now let $x\in X$ be a point. If $x$ does not lie in any $D_i$, take a smooth chart 
$f: B_\epsilon(0)\subset \RR^4 \to U$ to a neighbourhood 
$U$ of $x$ not touching any $D_i$. Then we consider the orbifold chart $( B_\epsilon(0),f,\{1\})$.

If $x$ lies in only one $D=D_i$ with $m=m_i$, take a chart as follows. 
Take a small neighbourhood $V\subset D$ of $x$, and by using coordinates we identify
$V\subset \RR^2$. Consider the exponential map
from the normal bundle (on $V$) $N_{D}$ to $X$, $\exp: N_{D}\to X$. For small $\epsilon>0$, $\exp: N_D^\epsilon=\{(x,v)| x\in V,
v\in (T_xD)^\perp, |v|<\epsilon\} \to X$ is a diffeomorphism onto its image. 
Trivialize the normal bundle, so that $N_D^\epsilon \cong V\x B_\epsilon(0)$. 
This gives a smooth chart $f:V\x B_\epsilon(0) \to U$, $f(w_1,w_2)=\exp_{w_1}(w_2)$,
with coordinates $(w_1,w_2)$. 
We define the following orbifold chart: consider $\tilde U=V \x B_{\epsilon}(0)$ and 
$\phi: \tilde U=V\x B_{\epsilon}(0) \to U$, by
$\phi(z_1,z_2) = f (z_1, re^{2\pi m i\theta})$, for $z_2=re^{2\pi i\theta}$.
The action of $\ZZ_{m}$ is given by
$\xi\cdot (z_1,z_2)=(z_1,\xi z_2)$, $\xi=e^{2\pi i /m}$. This defines a chart $(\tilde U, \phi, \ZZ_{m})$ at $x$.

If $x$ lies in the intersection of two surfaces, say $D_1,D_2$, with coefficients $m_1,m_2$, then $\gcd(m_1,m_2)=1$,
by assumption. Take small neighbourhoods $V_1\subset D_1$, $V_2\subset D_2$ of $x$, which we identify
with balls $B_\epsilon(0)\subset \RR^2$. Consider a smooth chart 
$f:B_{\epsilon}(0)\x B_{\epsilon}(0) \subset \RR^2\x\RR^2 \to U$, with $f(0,0)=x$,
$D_1\cap U=f(B_{\epsilon}(0) \x \{0\})$, $D_2\cap U=f(\{0\}\x B_{\epsilon}(0))$, 
and $g$ is the standard metric on $U$. We define the orbifold chart as follows:
consider $\tilde U=B_{\epsilon}(0) \x B_{\epsilon}(0)$ and
$\phi: \tilde U \to U$, $\phi(z_1,z_2) = \phi(r_1e^{2\pi i\theta_1}, r_2e^{2\pi i\theta_2})
 = f(r_1e^{2\pi im_2\theta_1}, r_2e^{2\pi im_1\theta_2})$.
The action of $\ZZ_m= \ZZ_{m_2}\x \ZZ_{m_1}$, $m=m_1m_2$, 
is given by $\xi \cdot (z_1,z_2)=(\xi^{m_1} z_1,\xi^{m_2} z_2)$, where $\xi=e^{2\pi i/m}$.
Then $(\tilde U, \phi,\ZZ_m)$ is a chart at $x$.

It is easy to see that these charts are compatible with the $\ZZ_m$-actions. We only need to check this near a surface $D$, where the change of charts are those of the respective normal bundle $N_D$. So, if the changes of charts for the smooth structure of $X$ near the $D$ are of the form $(w_1, w_2) \mapsto (w_1',w_2')=(\varphi(w_1), h(w_1)w_2)$, for some smooth maps
$\varphi:V_\alpha\subset \CC \to V_\beta \subset \CC$, 
$h:V_\alpha \to S^1$, then the changes of charts for the orbifold structure of $X$ are of 
the form $(z_1, z_2) \mapsto (z_1',z_2')=(\varphi(z_1), h(z_1)^{1/m} z_2)$, for some $m$-th root of $h$, and so they are smooth and $\ZZ_m$-equivariant.
\end{proof}

Note that we have not introduced the freedom of choosing coefficients $j_i$ for each $D_i$. 
We claim that one can always arrange so that the local actions are as above.
First take a neighborhood that intersects just one isotropy surface $D_i$. As $\gcd(j_i,m_i)=1$, changing the generator of $\ZZ_{m_i}$, we can arrange that the action of $\ZZ_{m_i}$ has $j_i=1$, so it is given by $\xi \cdot (z_1,z_2)=(z_1,\xi z_2)$.
Finally, whenever $D_i,D_j$ intersect, as $\gcd(m_i,m_j)=1$, we can arrange simultaneously $j_i=m_i, j_j=m_j$
at $x$ (by (b) of Proposition \ref{prop:models}). This means that using a different set of $j_i$ does not
change the resulting orbifold. 

Also a smooth (cyclic, oriented) $4$-orbifold $X$ can be converted into a smooth manifold with the 
same underlying space such that the
isotropy surfaces are embedded submanifolds intersecting transversely. As we shall not use this 
construction, we do not include the proof.

\medskip

Let $X$ be an orbifold with atlas $\{(\tilde U_\alpha,\phi_\alpha,\Gamma_\alpha)\}$. An orbi-tensor on $X$ is a
collection of tensors $T_\alpha$ on each $\tilde U_\alpha$ which are $\Gamma_\alpha$-equivariant, and which
agree under the changes of charts. In particular, we have orbi-differential forms $\Omega_{orb}^p(X)$, orbi-Riemannian
metrics $g$, and orbi-almost complex structures $J$. The exterior differential, covariant derivatives, Lie bracket, 
Nijenhuis tensor, etc, are defined in the usual fashion.

\begin{definition} \label{def:orb-sympl}
 A symplectic orbifold $(X,\omega)$ is an orbifold $X$ with a $\omega\in \Omega^2_{orb}(X)$ such that 
$d\omega=0$ and $\omega^n>0$, where $2n=\dim X$.

An almost K\"ahler orbifold $(X,J,\omega)$ consists of an orbifold $X$, and orbi-almost complex structure $J$ and
an orbi-symplectic form $\omega$ such that $g(u,v)=\omega(u,Jv)$ defines an orbi-Riemannian metric.

A K\"ahler orbifold is an almost K\"ahler orbifold satisfying the integrability condition that the Nijenhuis tensor $N_J=0$.
This is equivalent to requiring that the changes of charts are biholomorphisms of open sets of $\CC^n$.
\end{definition}

The following (presumably well-known) result will be useful in the following. It allows to have a nice local 
picture of the intersection of two symplectic surfaces in a symplectic $4$-manifold.

\begin{lemma} \label{lem:symplectic orthogonal}
Let $(X,\omega)$ be a symplectic 4-manifold, and suppose that $S, N \subset X$ are symplectic surfaces 
intersecting transversely and positively. Then we can perturb $S$ (small in the $C^0$-sense) 
so that $S$ is symplectic, $S$ and $N$ 
intersect $\omega$-orthogonally, and the perturbation only changes $S$ near the points of intersection with $N$. 

Moreover, once we perturb $S$, there are Darboux coordinates $(z,w)$ near all the intersection points of $N$ and $S$ 
in which $N=\{z=0\}$ and $S=\{w=0\}$.
\end{lemma}

\begin{proof}

We can arrange that the intersection becomes orthogonal after a small symplectic
isotopy around the intersection point. Suppose we are working in a Darboux chart $(z,w)$
with symplectic form $\omega=-\frac{i}{2} (dz\wedge d\bar z+dw\wedge d\bar w)$, let $N$
be given by the equation $N=\{z=0\}$, and $S$ be given as the graph of a map
$w=az+b\bar{z}+ g(z)$, where $a,b\in \CC$, $|g(z)|\leq C|z|^2$, $|\partial_z g(z)| + |\partial_{\bar{z}} g(z)| \leq C|z|$. 
The condition for $S$ to be symplectic and intersect positively $N$ is that $|a|^2-|b|^2 + 1 > 0$.

One can deform $S$ locally to 
 $$
  S'= \left\{ \left(z,\rho \left( ({|z|}/{\epsilon})^{2\alpha}\right) (az+b\bar z+g(z) ) \right) \right\},
  $$
for some $\epsilon>0$ and $\alpha>0$ to be determined later, where $\rho(t)$ is a bump function which is $0$ on 
$[0,1]$ and $1$ on $[2,\infty)$. Clearly $S'$ intersects $N$ at $(0,0)$ orthogonally with
respect to $\omega$. An easy calculation gives that
\begin{align} \label{eqn:eqn}
 \omega|_{S'} &= - \frac{i}2 (dz\wedge d\bar{z}+dw\wedge d\bar w)  \\
  &= \left(1+ \left(2\alpha\rho\rho' \cdot \left(\frac{|z|}{\epsilon}\right)^{2\alpha} 
+ \rho^2\right)(|a|^2-|b|^2) + O
 \left(\frac{|z|^{2\alpha+1}}{\epsilon^{2\alpha}}+|z| \right) \right) \frac{-i dz\wedge d\bar{z}}2,  \nonumber
\end{align}
where $\rho=\rho(({|z|}/{\epsilon})^{2\alpha})$, 
$\rho'=\rho' (({|z|}/{\epsilon})^{2\alpha})$. Now
 $$
  2\alpha\rho\rho' \cdot \left(\frac{|z|}{\epsilon}\right)^{2\alpha} \leq  \alpha C \, 2^{2\alpha+1}<\delta,
 $$
for any small $\delta>0$, by choosing $\alpha$ small enough. If $|a|^2-|b|^2 \geq 0$, clearly
  $$
 1+ \left(2\alpha\rho\rho' \cdot \left(\frac{|z|}{\epsilon}\right)^{2\alpha} + \rho^2 \right)(|a|^2-|b|^2)  > 0.
$$
If $0>|a|^2-|b|^2 >-1$, then
  $$
 1+ \left(2\alpha\rho\rho' \cdot \left(\frac{|z|}{\epsilon}\right)^{2\alpha} + \rho^2 \right) (|a|^2-|b|^2) 
 > 1 +  (\delta+1) (|a|^2-|b|^2)>0,
$$
choosing $\delta>0$ small enough. The error term in (\ref{eqn:eqn}) is $O(\epsilon)$, so 
it can be neglected for $\epsilon$ small enough. This completes the proof that $S'$ is a symplectic surface.
Clearly, close to $(0,0)$, $S'$ is given by the equation $w=0$.
\end{proof}

For constructing symplectic orbifolds, we have the following result. Again, this seems to be fairly well-known, but
we have not been able to find a proof in the literature.

\begin{proposition} \label{prop:orb->symp}
Let $X$ be a symplectic smooth (oriented) $4$-manifold with symplectic surfaces 
$D_i$ intersecting transversely and positively, and coefficients $m_i>1$ such that
$\gcd(m_i,m_j)=1$ if $D_i$, $D_j$ intersect. 
Then there is a smooth symplectic orbifold $X$ with 
isotropy surfaces $D_i$ of multiplicities $m_i$.
\end{proposition}

\begin{proof}
By Lemma \ref{lem:symplectic orthogonal}, we can assume that the surfaces $D_i$ intersect orthogonally.
As in the proof of Proposition \ref{prop:smooth->orb}, we start by fixing a metric. We do this as follows. First 
at each point at an intersection $D_i\cap D_j$, fix a Darboux chart $f: B_{\epsilon}(0) \x
B_{\epsilon}(0)\to U$ with $D_i\cap U=f(B_{\epsilon}(0) \x\{0\})$ and $D_j\cap U=f(\{0\}\x
B_{\epsilon}(0))$. 
Take a standard metric on $U$, and the corresponding almost complex structure $J_U$ on $U$. 
Fix now compatible almost complex structures $J_i$ on each $D_i$ (that is, $J_i:T_xD_i\to T_xD_i$
at each $x\in D_i$), which agree on $U$ with $J_U$. 
The normal bundle $N_{D_i}$ over $D_i$ is a symplectic bundle. Take a Riemannian metric
on $N_{D_i}$ compatible with its symplectic structure, and define a Riemannian metric on each
$T_x X=T_x D_i\oplus N_{D_i,x}$, $x\in D_i$, by declaring the direct sum orthogonal. We extend this metric $g$ on
$\bigcup D_i$ to a Riemannian metric on the whole of $X$ compatible with the symplectic form. 
This produces an almost K\"ahler structure on the whole of $X$ for which each $D_i$ is a $J$-invariant surface.

Now we use this metric $g$ for producing the atlas of Proposition \ref{prop:smooth->orb} that gives
$X$ the structure of a smooth orbifold. Let us now construct the orbifold symplectic form. We need first to modify
$\omega$ to a nearby $\omega'$ as follows. 

Let $D=D_i$ be one of the isotropy surfaces. On $N_{D}^\epsilon$ we have a radial coordinate $r$, 
and an angular coordinate $\theta$, well-defined in every chart up to addition of a function on $D$. 
By construction, we have $\omega=\omega|_D  + r\, dr\wedge d\theta$ along $D$.
For the bundle $N_D\to D$, consider a connection $1$-form $\eta \in \Omega^1(N_D - D)$, and let $F=d\eta \in \Omega^2(N_D)$ be its
curvarture. Thus  $\Omega=r dr\wedge \eta -\frac12 r^2 F+\omega|_D$, is a closed form on $N_D$ that 
coincides with $\omega$ along $D$. 
In the last expression, $\omega|_D$ stands for the pull-back of $\omega|_D$ by the bundle projection. 
Now $|\Omega-\omega|\leq Cr$, where $C$ is a constant independent of $r$. On $N_D^\epsilon$, 
$\Omega-\omega$ is closed so (being zero on $D$) it is exact, say $\Omega-\omega = d\beta$.

We can choose the $1$-form  $\beta$ so that it satisfies $|\beta|\leq Cr^2$, by the
usual standard procedure to produce a primitive of an exact form. Indeed, if 
$\Omega-\omega=\alpha_0 \wedge dr + \alpha_1$, one takes 
$\beta= \int_0^r \alpha_0 dr$, 
which is smooth (see \cite[p.\ 542]{Gompf}).

We also arrange the $1$-form $\eta$ to be equal to $d\theta$ on $U\cap N_D^\epsilon$, so that $F=0$ 
on $U\cap N_D^\epsilon$ and so
$\Omega=\omega$ on $U\cap N_D^\epsilon$.These forms $\Omega$'s for the different $D$'s paste to a globally defined 
$\Omega$ on a neighbourhood of $\bigcup D_i$.

Take a cut-off function $\rho:[0,\epsilon]\to [0,1]$ with $\rho(r)\equiv 1$ for $r\in [0,\frac13 \epsilon]$,
and $\rho(r)\equiv 0$ for $r\in [\frac23\epsilon,\epsilon)$, and $|\rho'|\leq C/\epsilon$. Hence
$\omega'=\omega+d(\rho\beta)$ satisfies that it is equal to $\Omega$ for $|r|\leq \frac13\epsilon$,
equal to $\omega$ for $|r|\geq \frac23\epsilon$, and $|\omega'-\omega| =
|d(\rho\beta)|=|d\rho \wedge \beta+ \rho \wedge d\beta| \leq C\epsilon$. This 
produces a globally defined $2$-form $\omega'$ on $X$. For $\epsilon$ small enough, $\omega'$ is symplectic.

Now let us define our orbi-symplectic form. Take first a point $x$ in some $D=D_i$ and not in $U$. We have
smooth coordinates $(w_1,w_2)$, $w_2=re^{2\pi i\theta}$, and 
orbifold coordinates $(z_1,z_2)$, $z_1=w_1$ and $z_2=re^{2\pi i\vartheta}$, $\theta=m\vartheta$. The action is $\xi\cdot (z_1,z_2)=(z_1,\xi z_2)$.
Here $\omega'=\Omega=\alpha +r\, dr\wedge d\theta+r\, dr\wedge \gamma$, where $\alpha$ is a $2$-form
and $\gamma$ is a $1$-form, and both $\alpha$ and $\gamma$ are invariant in the fiber direction, in particular $\SO(2)$-equivariant (recall that the connection $1$-form is $\eta=d\theta +\gamma$). 

We set, in the orbifold coordinates $(z_1,r, \vartheta)$,
 $$
 \hat\omega=\alpha + m \, r\, dr\wedge d\vartheta+r\, dr\wedge \gamma.
 $$
This is closed, smooth, symplectic and $\ZZ_m$-invariant. Moreover, $\hat\omega$ agrees with the pull-back of $\omega'$ via the orbifold chart $(z_1,z_2) \mapsto (w_1,w_2)$, and this implies that $\hat\omega$ is invariant by the orbifold change of charts.

Finally, on $U$, we take smooth coordinates $(w_1,w_2)$,  $w_1=r_1e^{2\pi i\theta_1}$, 
$w_2=r_2e^{2\pi i\theta_2}$, and orbifold coordinates are 
$z_1=r_1e^{2\pi i\vartheta_1}$,  $z_2=r_2e^{2\pi i\vartheta_2}$, with $\theta_1=m_2\vartheta_1$, $\theta_2=m_1
 \vartheta_2$. Here $\omega'=r_1\, dr_1\wedge d\theta_1+r_2\, dr_2\wedge d\theta_2$. We set
 $$
  \hat\omega= m_2 r_1\, dr_1\wedge d\vartheta_1+ m_1 r_2\, dr_2\wedge d\vartheta_2\, ,
 $$
which defines an orbifold symplectic form on $U$. 
\end{proof}

\begin{remark} \label{rem:omega'}
Consider the orbifold forms $(\Omega_{orb}(X),d)$. Their cohomology is denoted $H_{orb}^*(X)$. 
This is isomorphic to the usual De Rham cohomology \cite[p.\ 8]{CFM}, 
$H^*_{orb}(X)\cong H_{DR}^*(X)$.  The isomorphism can be explicitly constructed as follows: take
a smooth map $\varphi: X \to X$ such that it is the identity off a neighbourhood of $\bigcup D_i$,
and contracts radially a smaller neighbourhood of each $D_i$ to $D_i$, followed by a map that
contracts a neighbourhood of each point in an intersection $D_i\cap D_j$ to the point. Then
the map $\varphi^*: \Omega^*_{orb}(X) \to \Omega^*(X)$ gives the isomorphism
$\varphi^*:H^*_{orb}(X)\to H_{DR}^*(X)$.

The orbifold form $\hat\omega$ defines a class in $H^2_{orb}(X)$
and this is $[\hat\omega]=[\omega']$ under the above isomorphism. 
\end{remark}

\begin{lemma} \label{lem:admits-aK}
Let $(X,\omega)$ be a symplectic orbifold. Then $(X,\omega)$ admits the structure of an almost K\"ahler orbifold.
\end{lemma}

\begin{proof}
We have to adapt the usual construction of an almost K\"ahler structure for a symplectic manifold. 
This can be found in \cite[p.\ 68]{Silva}. Choose an orbifold metric $g'$, and define the orbifold $(1,1)$-tensor
$A$ by $g'(AX,Y)=\omega(X,Y)$. Then $AA^*$ is positive definite and symmetric and hence it is has the well-defined
square root $\sqrt{AA^*}$. Thus, it is an orbifold section of $\End_{orb} (TM)$. Then  $J= (\sqrt{AA^*})^{-1} A$
is an orbifold $(1,1)$-tensor (i.e.\ $J\in \End_{orb} (TM)$) and it clearly satisfies $J^2=-\id$. Recall that
$\omega(JX,JY)=\omega(X,Y)$, and $\omega(X,JX)=g'(\sqrt{AA^*}X,X)$, so the orbi-metric $g$ 
 compatible with $(\omega, J)$ is $g(X,Y)=g'(\sqrt{AA^*}X,Y)$. 

Since this local construction of $J$ and $g$ is canonical (once we have chosen the orbi-metric $g'$), it is compatible with change of charts, so this defines an orbifold almost K\"ahler structure.
\end{proof}

\begin{proposition}\label{prop:locus-orb-K}
If $X$ is a smooth K\"ahler cyclic orbifold, then $X$ is a smooth complex manifold and $D_i$ are complex curves
intersecting transversely.
\end{proposition}

\begin{proof}
As the  almost-complex structure is integrable, we can take the orbifolds charts 
$\phi:\tilde U\to U\subset X$ to be holomorphic, with $\tilde U\subset \CC^2$. The group $\Gamma=\ZZ_m$ acts by
a biholomorphism $f:\tilde U\to \tilde U$. The map $\varphi(z)= \frac1m \sum_{k=0}^{m-1} f^k( d_0 f^{-k}(z))$
defines a new chart $\phi'=\phi\circ \varphi$, where the action of $\Gamma$ is linear,
since $\varphi(d_0f(z))=f(\varphi (z))$.

So $\Gamma<\GL(2,\CC)$ acts by complex transformations, and the quotient $\tilde U/\Gamma$ has a 
natural complex structure (that is,
the complex structure on the complement of $\bigcup D_i$ extends naturally to $\bigcup D_i$). 
The induced map $\bar\phi:\tilde U/\Gamma \to U$ is holomorphic, and thus biholomorphic since it is bijective.
These maps define an atlas as a complex manifold. Note that if $(z_1,z_2)$ are the coordinates for
an orbifold holomophic chart, with action $\xi \cdot (z_1,z_2)=(\xi^{m_2} z_1,\xi^{m_1} z_2)$, $\xi=
e^{2\pi i/m}$, $m=m_1m_2$, then $w_1=z_1^{m_1}$, $w_2=z_2^{m_2}$ define holomorphic
coordinates for the quotient. The surfaces $D_i$ are defined by the equations
$w_1=0$ or $w_2=0$ in such charts, therefore they are smooth complex curves intersecting transversely.
\end{proof}

\section{Seifert bundles}\label{sec:3}

A Seifert bundle is a space fibered by circles over an orbifold. We give a precise definition.

\begin{definition} \label{definition seifert bundle}
Let $X$ be a cyclic, oriented $n$-dimensional orbifold. A Seifert bundle over $X$ is an oriented  
$(n+1)$-dimensional manifold $M$ equipped with a smooth $S^1$-action and a 
continuous map $\pi:M \to X$ such that for an orbifold chart $(\tilde U, \phi, \ZZ_m)$, there is
is a commutative diagram
 $$
\begin{tikzcd}
(S^1 \times \tilde U)/\ZZ_{m} \arrow{r}{\cong} \arrow{d}{\pi} & \pi^{-1}(U) \arrow{d}{\pi} \\
\tilde U/\ZZ_{m} \arrow{r}{\cong} & U
\end{tikzcd}
 $$
where the action of $\ZZ_{m}$ on $S^1$ is by multiplication by $\xi=e^{2\pi i /m}$ and the top diffeomorphism is
$S^1$-equivariant.
\end{definition}

\begin{proposition} \label{prop:Seifert}
 An oriented $(n+1)$-manifold endowed with a fixed point free action of $S^1$ is a Seifert bundle over
 a cyclic, oriented $n$-orbifold.
\end{proposition}

\begin{proof}
 Let $M$ be a manifold  endowed with a fixed point free action of $S^1$. Then $X$ will be the space of leaves of the 
$S^1$-action. The orbifold structure on $X$ is obtained as follows. Take an auxiliary Riemannian metric $g$ and
average it over $S^1$ to make it $S^1$-invariant. For a point $p\in M$, let $O(p)$ be the orbit of $p$. Let $I(p)=\ZZ_m=
\la \xi\ra$, $\xi=e^{2\pi i/m}$, be
the isotropy of $p$. Then the action of $\xi$, say $f:M\to M$, fixes $p$ and the tangent direction $R_p$ to the orbit $O(p)$. Hence
the differential of $d_pf$ fixes the orthogonal hyperplane $H_p=R_p^{\perp}$, inducing an action of $\ZZ_m$ on it. Since
$M$ is oriented, $d_pf$ preserves orientation, so $\ZZ_m =\la d_pf\ra <\SO(n)$. 

For a small $\tilde U\subset H_p$, the exponential map and the $S^1$-action give a local diffeomorphism 
$\varphi: S^1\x \tilde U \to M$, $\varphi(u, z)=u\cdot \exp_p(z)$. On $S^1\x \{0\}$ the isotropy is $\ZZ_m$, 
hence a neighbourhood of $O(p)=\varphi(S^1\x \{0\})$ is modelled on $(S^1\x\tilde U)/\ZZ_m$. 
This action is by multiplication by $\xi$ on the $S^1$-factor, and
by the action of $d_pf$ on $\tilde U$. The space of leaves is identified with $\tilde U/\ZZ_m$, 
and $(\tilde U, \bar\varphi,\ZZ_m)$ gives the desired orbifold chart, $\bar\varphi:\tilde U \to \tilde U/\ZZ_m \subset X$.
\end{proof}

Suppose in the following that $X$ is a $4$-dimensional orbifold and $\pi: M\to X$ is a Seifert bundle over $X$.
According to the normal form of the $\ZZ_m$-action given in (\ref{action}), the open subset
$\pi^{-1}(U)\cong (S^1\times \tilde U)/\ZZ_m$
is parametrized by $(u,z_1,z_2)\in S^1\x \CC^2$, modulo the $\ZZ_m$-action $\xi\cdot (u,z_1,z_2) 
=(\xi u, \xi^{j_1} z_1, \xi^{j_2}z_2)$, for some integers $j_1,j_2$, where
$\xi=e^{2\pi i /m}$. The $S^1$-action is given by
$s\cdot (u,z_1,z_2)=(su,z_1,z_2)$, so $\ZZ_m \subset S^1$ is the isotropy group of $O(p)\subset M$, 
and the exponents $j_1,j_2$ are determined by the $S^1$-action.

We say that $\{(D_i,m_i,j_i)\}$ are the \emph{orbit invariants} of the Seifert bundle if $D_i\subset X$ are
the isotropy surfaces, with multiplicities $m_i$, and the local model around a point $p\in  
D_i^o=D_i - \bigcup_{i\neq j} (D_i \cap D_j)$ is of the form 
$(S^1\times \tilde U)/\ZZ_{m_i}$ with action $\xi\cdot (u,z_1,z_2) =(\xi u, z_1, \xi^{j_i} z_2)$, 
$D_i=\{z_2=0\}$. If the orbifold is smooth, then 
for a point $p\in D_i\cap D_j$, the local model is of the form
$(S^1\times \tilde U)/\ZZ_{m_i}$ with action $\xi\cdot (u,z_1,z_2) =(\xi u, \xi^{j_j}z_1, \xi^{j_i} z_2)$, 
$D_i=\{z_2=0\}$, $D_j=\{z_1=0\}$.

\begin{definition}
For a Seifert bundle $\pi:M\to X$, we define its Chern class as follows. Let $\mu=\ZZ_{m(X)}$, where
$m(X)=\lcm\{ m(x) \, | \, x\in X\}$. Consdier the circle fiber bundle $M/ \mu \to X$ and its 
Chern class $c_1(M/\mu)\in H^2(X,\ZZ)$. We define
 $$
 c_1(M/X)=\frac{1}{m(X)} c_1(M/\mu) \in H^2(X,\QQ).
 $$
\end{definition}

The next proposition shows that the orbit invariants determine the Seifert bundle globally when $X$ is smooth.

\begin{proposition} \label{prop:seifert-existence}
Let $X$ be an oriented $4$-manifold and $D_i \subset X$ oriented surfaces of $X$ which
intersect transversely. Let $m_i>1$ such that $\gcd(m_i,m_j)=1$ if $D_i$ and $D_j$ intersect. 
Let $0<j_i<m_i$ with $\gcd(j_i,m_i)=1$ for every $i$. Let $0<b_i<m_i$ such that $j_ib_i\equiv 1 \pmod{m_i}$.
Finally, let $B$ be a complex line bundle on $X$. Then
there is a Seifert bundle $f:M  \to X$ with orbit invariants $\{(D_i,m_i,j_i)\}$ and first Chern class 
 \begin{equation}\label{eqn:c1}
 c_1(M/X)=c_1(B) + \sum_i \frac{b_i}{m_i} [D_i]. 
 \end{equation}
The set of all such Seifert bundles forms a principal homogeneous space under $H^2(X,\ZZ)$, 
where the action corresponds to changing $B$. 
\end{proposition}

\begin{proof}
The orbit invariants determine uniquely the structure of smooth orbifold for $X$.
Let $\{(\tilde U_\alpha,\phi_\alpha,\Gamma_\alpha)\}$ be a covering of $X$ by orbifold charts.
The orbit invariants determine the local model of the Seifert bundle, so the possible 
Seifert bundles with these orbifold invariants are given by the gluing of the local models. 
This is defined by transition functions $g_{\alpha\beta}:\tilde{U}_\alpha \cap
\tilde{U}_{\beta}\to S^1$ which are $\Gamma_\gamma$-invariant for $\tilde{U}_\gamma\subset \tilde{U}_\alpha \cap
\tilde{U}_{\beta}$. Therefore it is defined by a $1$-cocycle 
in $\SC^\infty_{orb}(S^1)$, the orbifold functions with values in $S^1$. 
Using the exponential short exact sequence of sheaves $0\to \ZZ \to \SC^\infty_{orb} \to \SC^\infty_{orb}(S^1) \to 0$,
where the sheaf of orbifold real functions $\SC^\infty_{orb}$ is a fine sheaf, we have that
the possible Seifert bundles are parametrized by $H^1(X,\SC^\infty_{orb}(S^1)) \cong H^2(X,\ZZ)$.
We can tensor $M\ox B$, for a line bundle $B\to X$, by multiplying the transition
functions. Therefore the set of Seibert bundles forms a homogeneous space under $H^2(X,\ZZ)$.

For $M\ox B$, we have that $(M\ox B)/\mu=(M/\mu) \ox B^{\ox m}$, since the quotient is given locally by
$(u,z_1,z_2) \mapsto (u^m,z_1,z_2)$, where $m=m(X)$. So 
 $$
 c_1((M\ox B)/X)=\frac1m c_1((M\ox B)/\mu)=\frac1m (c_1(M/\mu) +m \, c_1(B))=
 c_1(M/X)+ c_1(B).
 $$

To prove  (\ref{eqn:c1}) is equivalent to prove that $c_1(M/\mu)=m \, c_1(M/X) \equiv \sum_i 
b_i\frac{m}{m_i}[D_i] \pmod{m}$. For this we take a $2$-cycle $S\subset X$, that we can assume
that it intersects transversely the $D_j$'s, and compute
$\la c_1(M/\mu),S\ra$. To compute $c_1(M/\mu)$, we fix a transverse section $s$ of the line bundle
associated to $M/\mu$.

In $\la c_1(M/\mu),S\ra$ there is a contribution coming from balls $B_p\subset S$ around each 
intersection point $p\in S\cap (\bigcup D_i)$ and a contribution from $S^o=S-\bigcup B_p$. 
The second one is $\la c_1(M/\mu),S^o\ra=\la m \,c_1(M/X),S^o\ra \in m\,\ZZ$, since $M\to X$ is
an honest circle bundle over the locus $S^o$. For this equality we choose $s$ to be the image
of a section of the line bundle associated to the circle bundle $M\to S^o$.

Now we look at the circle bundle $M/\mu \to X$
at a point $p\in S\cap D_i$. We can arrange orbifold coordinates $(z_1,z_2)$ such that $D_i=\{z_2=0\}$
and $S=\{z_1=0\}$. The Seifert bundle is given by coordinates $(u,z_1,z_2)$ modulo
$\xi\cdot(u,z_1,z_2)= (\xi u,z_1,\xi^{j_i} z_2)$, $\xi=e^{2\pi i/m_i}$. Equivalently, modulo
$(u,z_1,z_2) \mapsto (\xi^{b_i} u,z_1,\xi  z_2)$. The circle bundle $M/\mu$ is parametrized
by $(v=u^m, z_1,z_2)$ modulo $(v,z_1,z_2) \mapsto (v,z_1,\xi  z_2)$. 
The section $s$ lifts to a section $\hat s$ of $M$ over $\bd B_p$. In 
orbifold coordinates of $X$, it is of the form $\hat{s}(z_1,z_2) = (u(z_1,z_2),z_1,z_2)$ ,
with $u(z_1,\xi z_2)= \xi^{b_i} u(z_1, z_2)$. This means that we can choose $u(z_1,z_2)=z_2^{b_i}$.
Therefore, the section $s$ is locally $s(z_1,z_2)=(v,z_1,z_2)$ with $v=z_2^{b_i m}$.
Going back to smooth coordinates $w_1=z_1$, $w_2=z_2^{m_i}$, the section is written as
$s(w_1,w_2) =(w_2^{b_im/m_i},w_1,w_2)$. Therefore the
zero set of $s$ along $S=\{w_1=0\}$ has multiplicity ${b_im}/{m_i}$.
Adding the contributions of all points $p \in S\cap D_i$, we get the contribution $\sum \frac{b_im}{m_i} \la [D_i],S\ra$ to 
$\la c_1(M/\mu),S\ra$. This proves the sought formula.
\end{proof}

Let $\pi:M\to X$ be a Seifert bundle, $p\in M$ and $x=\pi(p)$.
The fiber over $x$ is the orbit $O(p)$, which is of the form $S^1/\ZZ_m$, where $m=m(x)=m(p)$ is both the 
isotropy of $x$ (as orbifold point) and the isotropy of $p$ (for the $S^1$-action on $M$). 
We call the orbit $O(p)$ \emph{semi-regular} if the orbifold point $x=\pi(p)$ is smooth.
This means that the local model in Proposition \ref{prop:models} is of type (b) or (d). In the case (d), the orbit
$O(p)$ has nearby orbits $O(p')$ with multiplicity $m(p')=m(p)$. In case (b), $m=m_1m_2$ and $\gcd(m_1,m_2)=1$,
and $O(p)$ has nearby orbits $O(p_1)$ and $O(p_2)$ of multiplicities $m_1,m_2$, respectively.

\begin{definition}
We say that a Seifert bundle is \emph{semi-regular} if the base orbifold $X$ is smooth, that is all orbits are semi-regular.
\end{definition}

Now we want to relate the homology of $M$ with that of $X$ for a Seifert bundle $\pi:M\to X$. We only need the case
of a semi-regular Seifert bundle, and we are interested in the case where $H_1(M,\ZZ)=0$. 
We have the following result.

\begin{theorem} \label{thm:Kollar}
Suppose that $\pi:M\to X$ is a semi-regular Seifert bundle with isotropy surfaces $D_i$ with multiplicities $m_i$. 
Then $H_1(M,\ZZ)=0$ if and only if
 \begin{enumerate}
 \item $H_1(X,\ZZ)=0$,
 \item $H^2(X,\ZZ)\to \sum H^2(D_i,\ZZ/m_i)$ is surjective,
 \item $c_1(M/\mu)\in H^2(X,\ZZ)$ is primitive.
 \end{enumerate}
 Moreover, $H_2(M,\ZZ)=\ZZ^k\oplus \bigoplus (\ZZ/m_i)^{2g_i}$, $g_i=$genus of $D_i$, $k+1=b_2(X)$.
\end{theorem}

\begin{proof}
First, suppose that $X$ is smooth and satisfies (1)--(3). We have $H^3(X,\ZZ)=H_1(X,\ZZ)=0$, 
by Poincar\'e duality. By \cite[Proposition 26(3)]{Kollar}, we get $H_1(M,\ZZ)=0$. Now \cite[Corollary 27]{Kollar}
gives that $b_2(M)=k$ and \cite[Proposition 28]{Kollar} gives that $H_2(M,\ZZ)_{tors}=\bigoplus (\ZZ/m_i)^{2g_i}$.

Conversely, if $H_1(M,\ZZ)=0$ then the argument in \cite[\S25]{Kollar} gives that $H_1(X,\ZZ)=0$. 
Then  \cite[Proposition 26(3)]{Kollar} implies conditions (2)--(3).
\end{proof}

\begin{remark} \label{rem:e2}
If $X$ is a smooth $4$-manifold and $\pi:M\to X$ is a circle bundle, 
then Theorem \ref{thm:Kollar} also applies, just taking empty isotropy locus. Then
$H_1(M,\ZZ)=0$ if and only if $H_1(X,\ZZ)=0$ and $c_1(M)$ is primitive, and in that case
$H_2(M,\ZZ)=\ZZ^k$ with $k=b_2(X)-1$. Note that $H_2(M,\ZZ)$ is torsion-free.
\end{remark}

\begin{corollary} \label{cor:Kollar}
Suppose that $M$ is a $5$-manifold with $H_1(M,\ZZ)=0$ and 
$H_2(M,\ZZ)=\ZZ^k\oplus \bigoplus_{i=1}^{k+1} (\ZZ/p^i)^{2g_i}$, $k\geq 0$, $p$ a prime, and $g_i\geq 1$. 
If $M\to X$ is a semi-regular Seifert bundle, then $H_1(X,\ZZ)=0$, $H_2(X,\ZZ)=\ZZ^{k+1}$, and 
the ramification locus has $k+1$ disjoint surfaces $D_i$ linearly independent in rational homology, and
of genus $g(D_i)= g_i$.
\end{corollary}

\begin{proof}
By Theorem \ref{thm:Kollar}, $H_1(X,\ZZ)=0$ and $H_2(X,\ZZ)=\ZZ^{k+1}$. Let $D_1,\ldots, D_m$
be the isotropy surfaces of positive genus, and let $D_{m+1},\ldots, D_l$ be the isotropy spheres with coefficients $m_j$. 
For each coefficient $p^i$, $1\leq i\leq k+1$, there is at least one 
$D_i$ with isotropy $p^i$, hence it must be $m\geq k+1$. 
Now recall that the map 
$$
\ZZ^{k+1} \to \sum_{i=1}^{k+1} H^2(D_i,\ZZ/p^i) + \sum_{j=k+2}^{m} H^2(D_j,\ZZ/p^{i_j}) + \sum_{j=m+1}^l
H^2(D_j,\ZZ/m_j)
$$ 
is surjective. In particular, it cannot be $m>k+1$,
so there are exactly $k+1$ surfaces $D_i$, and they have each a different isotropy coefficient $p^i$. 
It follows that $g(D_i)=g_i$, and that $D_i$ and $D_j$ are disjoint for $i \ne j$.

Let us see that the classes $[D_i]$ are independent. The above map is given by
\begin{eqnarray*}
H^2(X,\ZZ)=\ZZ^{k+1} & \too & \sum_{i=1}^{k+1} H^2(D_i,\ZZ/p^i) = \sum_{i=1}^{k+1} \ZZ/p^i \\
 \,  [S] &\mapsto & \left( [S] \cdot [D_i] \pmod{p^i} \right)
\end{eqnarray*}
As this map is surjective, for each $[D_i]$ there exists an element $[S_i] \in H^2(X,\ZZ)$ so that
$[S_i] \cdot [D_i] \equiv 1  \pmod{p^{i}}$ and $[S_i] \cdot [D_j] \equiv 0 \pmod{p^{j}}$ for $j \ne i$.
Thus $[S_i] \cdot [D_i] \equiv 1  \pmod{p}$ and $[S_i] \cdot [D_j] \equiv 0 \pmod{p}$ for $j \ne i$.
If the $[D_i]$ are not linearly independent then there exists integers
$b_i$ so that $\sum b_i [D_i]=0$. We can choose $b_i$ that are coprime.
Multiplying by $[S_j]$ we get $\sum b_i [D_i]\cdot [S_j]=0$, for $1\le j \le k+1$. Reducing modulo $p$,
we have $b_i \equiv 0 \pmod{p}$, which is a contradiction. 
\end{proof}

\section{K-contact and Sasakian $5$-manifolds}\label{sec:4}

A Sasakian or a K-contact structure on a compact manifold $M$ is called \textit{quasi-regular} if there is a positive
integer $\delta$ satisfying the condition that each point of $M$ has a neighbourhood
such that each leaf for $\xi$ passes through $U$ at most $\delta$ times. If
$\delta\,=\, 1$, then the Sasakian or K-contact structure is called \textit{regular} (see \cite[p.\ 188]{BG}).

%

A result of \cite{RUK} says that if $M$ admits a Sasakian structure, then it admits also a quasi-regular Sasakian structure. Also, if
a compact manifold $M$ admits a K-contact structure, it admits a quasi-regular contact structure \cite{MT}. 

\begin{theorem} \label{thm:otro}
Let $(M,\eta,\Phi,\xi,g)$ be a quasiregular K-contact manifold. Then the space of leaves $X$ has a natural structure of an almost K\"ahler cyclic orbifold where the projection $M \to X$ is a Seifert bundle.

Furthermore, if $(M,\eta,\Phi,\xi,g)$ is Sasakian, then $X$ is a K\"ahler orbifold.
\end{theorem}

\begin{proof}
Take a point $p \in M$, and let $O(p)$ be the orbit through $p$. Since $O(p)$ interesects finitely 
many times every small neighbourhood, then $O(p)$ must be a circle. Let 
$\phi_t$ be the Reeb flow, and consider $t_p$ the period of $\phi_t(p)$. Let $f=\phi_{t_p}$,
$H_p=\la \xi_p\ra^\perp=\ker \eta|_{p}$, and $d_pf:H_p\to H_p$. For $\epsilon >0$ small, take $B_\epsilon(0)\subset H_p$.
Then $\varphi:\RR\x B_\epsilon(0)\to X$, $\varphi(t,w)=\phi_t(\exp_p(w))$, is an open embedding whose image $W$
is a neighbourhood of $O(p)$ consisting of orbits of the Reeb flow (recall that the Reeb flow is by isometries, so it 
preserves the distances to $O(p)$). 
Since $\xi$ is a quasi-regular vector field, the orbits intersect $S_p=\varphi(\{0\}\x B_\epsilon(0))$ 
at finitely many points. For $q=\varphi(0,w)$, the points of intersection are $f^k(q)=\varphi(k\,t_p,w)$, $k\in \ZZ$.
So there is some $k$ such that $f^k(q)=q$, i.e.,\ $d_pf^k(w)=w$.
Therefore $d_pf:H_p\to H_p$ is of finite order. Let $m$ be its order. So $d_pf^m=\Id$, hence 
$f^m=\Id$. Therefore $\phi_t$ gives an $S^1$-action with period $m\,t_p$.

By Proposition  \ref{prop:Seifert}, we have a Seifert bundle $\pi:M\to X$, over the space of leaves $X$, which is
a cyclic orbifold. Let us see that $X$ has the structure of an almost K\"ahler orbifold. The open set $W\cong
(S^1\x B_\epsilon(0))/\ZZ_m$, and the orbifold chart is $(\tilde U=B_\epsilon(0),\phi, \ZZ_m)$, where
$\bar\varphi:B_\epsilon(0) \to X$, $\bar\varphi(w)=\pi(\varphi(0,w))=\pi(\exp_p(w))$. 
Then the orbifold tangent space at $\bar  p=\pi(p)$ is identified
with $T_0  \tilde U \cong H_p$. We put at $\bar p$ the complex 
and symplectic structures $J,\omega$ on $T_0 \tilde U$
given by $\Phi, d\eta$ on $H_p$, respectively. These are well defined independently of the point
in the orbit, since the Reeb flow acts by isometries, preserving $\Phi$ and $\eta$. Finally, these complex
and symplectic structures are $\ZZ_m$-invariant (since the action is given by $d_pf$, 
the isometry defined by the Reeb flow $f=\phi_{t_p}$).

Now suppose that $M$ is Sasakian. Then, by
definition, there is an integrable complex structure $I$ on the cone $C(M)=M\x\RR^{>0}$, given
by $I(X)=\Phi(X)$ on $\ker\eta$, and $I(\xi)=t\frac{\bd}{\bd t}$. This means that the Nijenhuis tensor vanishes, i.e.,
 \begin{equation}\label{eqn:NI} 
 N_I(X,Y)=-[X,Y]+ I[IX,Y]+ I[X,IY]-[IX,IY]=0.
 \end{equation} 
Take an orbifold chart $(\tilde U,\bar\varphi,\ZZ_m)$ as above with $W=(S^1\x \tilde U)/\ZZ_m$,
$\bar p=\bar\varphi(0)$. Take $X,Y$ two $\ZZ_m$-equivariant vector fields on $\tilde U$.
Let us see that $N_J(X,Y)_{\bar p}$ vanishes.
The vector fields $X,Y$ define vector fields, that we denote $X,Y$  again, on 
$W\x \RR^{>0} = ((S^1 \x \tilde U)/\ZZ_m )\x \RR^{>0}=
(S^1 \x \tilde U\x \RR^{>0})/\ZZ_m $ (defining them as zero along the
coordinates $S^1\x \RR^{>0}$). We write $X=X'+f\xi$, $Y=Y'+g\xi$, where 
$X'_x,Y'_x\in H_x$, for all points $x\in W$, $f,g$ smooth functions with $f(p)=g(p)=0$.
We expand $N_I(X,Y)$ given in (\ref{eqn:NI}), substitute at $p$, and discard the components
with $\xi, \partial_t$ (that is, project down to $H_p$). We get  
 $$
 -[X',Y']+ \Phi [\Phi X',Y']+\Phi [X',\Phi  Y']-[\Phi X',\Phi Y']=0,
 $$ 
at $p$. Using the projection $h: S^1 \x \tilde U \x \RR^{>0} \to \tilde U$,
and that the Lie bracket is preserved ($h_*[X',Y']=[X,Y]$), and the formula $h_*(\Phi X')=JX$, we get
$N_J(X,Y)=-[X,Y]+ J[JX,Y]+ J[X,JY]-[JX,JY]=0$ at $\bar p$, as required.
\end{proof}

\begin{lemma} \label{lem:omeg}
Let $(X,\omega)$ be a symplectic
$4$-manifold with a collection of embedded symplectic surfaces $D_i$
intersecting transverselly and positively, and integer numbers $m_i>1$, with $\gcd(m_i,m_j)=1$ 
whenever $D_i \cap D_j \neq \emptyset$. Then there is a Seifert bundle $\pi:M\to X$ such that:
\begin{enumerate}
\item It has Chern class $c_1(M/X)=[\hat\omega]$ for some orbifold symplectic form $\hat\omega$ on $X$.
\item If $\sum \frac{b_im}{m_i} [D_i]$ is primitive and the second Betti number $b_2(X)\geq 3$, then 
then we can further have that $c_1(M/\mu)\in H^2(X,\ZZ)$ is primitive.
\end{enumerate}
\end{lemma}

\begin{proof}
Consider the Seifert bundle $\pi:\widetilde{M} \to X$ given by some orbit invariants $\{(D_i,m_i,j_i)\}$,
$m=m(X)$, with $c_1(\widetilde{M}/\mu)=m\sum \frac{b_i}{m_i} [D_i]$, possible by 
Proposition \ref{prop:seifert-existence}. The set of elements
 \begin{equation}\label{eqn:aaa}
 \left\{\frac{1}{m k +1}a +\frac{1}{mk+1} c_1(\widetilde{M}/X) \, |\,  a\in  H^2(X,\ZZ),  k\geq 1 \right\}\subset H^2(X,\RR)
 \end{equation}
is dense. So we can perturb $\omega$ slightly so that 
$[\omega] = \frac{1}{m k +1}a+\frac{1}{mk+1} c_1(\widetilde{M}/X)$, 
for some $a\in  H^2(X,\ZZ)$ and $k\geq 1$.
Then the symplectic form $\tilde\omega=(mk+1)\omega$ satisfies that 
$[\tilde\omega]=a+c_1(\widetilde{M}/X)$. Choosing a line bundle $B$ with $c_1(B)=a$,
we have a Seifert bundle $M=\widetilde{M}\ox B$ with $c_1(M/X)=[\tilde\omega]$. 
Now the process of Proposition \ref{prop:orb->symp} gives an orbifold symplectic form $\hat\omega$ on
the orbifold $X$ with isotropy surfaces $D_i$ with multiplicities $m_i$. This has
$[\hat\omega]=[\tilde\omega]=c_1(M/X)$. This proves (1).

Now let us see (2).
Take a primitive class $b_1\in H^2(X,\ZZ)$ with $c_1(\widetilde{M}/\mu) \cdot b_1=0$. Then there
exists $a_0\in H^2(X,\ZZ)$ with $a_0\cdot b_1=1$. Now take a primitive $b_2 \in H^2(X,\ZZ)$ 
with $c_1(\widetilde{M}/\mu) \cdot b_2=0$ and $a_0\cdot b_2=0$, possible since $b_2(X)\geq 3$. 
Then let us see that the elements of (\ref{eqn:aaa}) with
$\gcd (a\cdot b_1,a\cdot b_2)=1$ are dense. Take any element $x$ in (\ref{eqn:aaa})
given by some $a$ and $k\geq 1$. Let $k_1=a\cdot b_1, k_2=a\cdot b_2 \in \ZZ$. Consider
$k_0$ a large integer containing all prime factors of both $k_1,k_2$. Then take
the element $x'$ given by $a'=k_0 a+a_0$, $k'=k_0 k$, which 
satisfies $|x'-x|\leq C|x|/k$. Note that we can suppose that $k$ is arbitrarily large in the expression of $x$.
Thus the set of such $x'$ is dense.

So consider an element $a$ with $\gcd(a\cdot b_1, a\cdot b_2)=1$ and a Seifert bundle $M$ with
$c_1(M/\mu)=ma+c_1(\widetilde{M}/\mu)=m[\tilde\omega]$ as above. Then 
$c_1(M/\mu) \cdot b_j= m(a\cdot b_j)$, $j=1,2$. Therefore if $c_1(M/\mu)$ is
divisible by some $\ell$, then $\ell | m$. So $c_1(\widetilde{M}/\mu)=c_1(M/\mu)-m a$ is divisible by 
$\ell$, and hence it is not a primitive class, contrary to hypothesis.
\end{proof}

Let $\pi:M\to X$ be any Seifert bundle.
We construct a connection $1$-form on $M\to X$ as follows. Take an orbifold covering
 $X= \bigcup U_\alpha$, with an  orbifold partition of unity $\{\rho_\alpha\}$.
For each $U_\alpha =\tilde U_\alpha/\ZZ_{m_\alpha}$ we have
$\pi^{-1}(U_\alpha)=(S^1\x \tilde U_\alpha)/\ZZ_{m_\alpha}$. Let
$\eta_\alpha=u_\alpha^{-1} d u_\alpha$, where $u_\alpha$ is the $S^1$-coordinate. Define
 $$
 \eta= \sum (\pi^*\rho_\alpha) \eta_\alpha.
 $$
This is an orbifold $1$-form and $F=d\eta=\sum d\rho_\alpha \wedge \eta_\alpha$
is the (orbifold) curvature $2$-form of $M\to X$.

For the circle fiber bundle $M/\mu \to X$, $\eta$ descends to a $1$-form $\bar\eta$ on
$M/\mu$. The fiber of $M/\mu$ is
parametrized by $\bar u_\alpha=u_\alpha^m$, $m=m(X)$.
So the connection $1$-form on $M/\mu$ equals $\bar\eta=m\eta$.
Its curvature is $m F$ and thus $c_1(M/\mu)=[m F]$. This implies that
 $$
 c_1(M/X)=\frac1m c_1(M/\mu)= [F].
 $$

The following result appears in \cite[p.\ 211]{BG}, where it is refered to \cite{Hat}. However
the proof in \cite{Hat} does not cover the orbifold case. So we have included a proof.

\begin{theorem} \label{thm:k-contact}
Let $(X,\omega,J,g)$ be an almost K\"ahler cyclic orbifold with 
$[\omega] \in H^2(X;\QQ)$, and let $\pi: M \to X$ be a Seifert 
bundle with $c_1(M/X)=[\omega]$. Then $M$ admits a K-contact structure 
$(\xi,\eta,\Phi,{g})$ such that $\pi^*(\omega)=d \eta$.
\end{theorem}

\begin{proof}
Take the (orbifold) connection $1$-form constructed above, and let 
$F=d\eta$ be its curvature. As $[F]=c_1(M/X)=[\omega]$, we have
that $F-\omega=d\beta$, for some orbifold $1$-form $\beta$. Then
we can change $\eta$ to $\eta'=\eta-\beta$, so that its curvature
is $F'=F-d\beta=\omega$.

Now the $1$-form $\eta$ is a smooth form on the total space $M$.
On each $\pi^{-1}(U)=(S^1\x\tilde U)/\ZZ_m$, we have that 
$d\eta=\omega$ is the $2$-form coming from $\tilde U$. So $\eta\wedge (d\eta)^2>0$,
and $\eta$ is a contact form. Now define the Reeb vector field $\xi$ as the
one given by the $S^1$-action, which clearly preserves $\eta$.
Define $H_p=\ker \eta_p$, and $\Phi:T_pM \to T_pM$ by
$\Phi(\xi)=0$ and $\Phi:H_p \to H_p$ 
as the almost complex structure $J_x:T_x\tilde U\to T_x\tilde U$, for $x=\pi(p)$,
under the isomorphism $H_p\cong T_x\tilde U$. This is well-defined since the
$S^1$-flow preserves the horizontal subspaces $H_p$. Clearly the Reeb flow
preserves $\Phi$.

Finally define the metric $g$ by declaring $H_p$ and $\xi_p$ orthogonal, $\xi_p$
unitary and $g$ is the metric on $H_p$ given by $\Phi$ and $\omega$. Then
the Reeb flow preserves $g$, i.e.,\ it acts by isometries. This means that
$(M,\xi,\eta,\Phi,{g})$ is a K-contact manifold.
\end{proof}

\begin{remark}\label{rem:Kim}
Theorem \ref{thm:k-contact} corrects a statement of \cite{Kim}, where
it is claimed that a K-contact structure can be constructed from an orbifold where
the isotropy locus is not a symplectic surface. This is wrongly used to construct examples
of K-contact manifold, and to conclude that the manifolds of \cite{Kollar} admit
K-contact structures, which is the main result of \cite{Kim}.
\end{remark}

\section{A symplectic $4$-manifold with many disjoint symplectic surfaces}\label{sec:5}

Now we move to the construction of a K-contact manifold which cannot admit a semi-regular Sasakian 
structure (Theorem \ref{thm:main}). For this, we need a symplectic manifold with many disjoint
symplectic surfaces which will be used to construct a Seifert bundle.

\begin{theorem}\label{thm:generators} There exists a simply connected symplectic $4$-manifold $X$ with 
$b_2=36$ and with $36$ disjoint surfaces $S_1, \ldots ,S_{36}$ such that
\begin{enumerate}
\item  $g(S_1)=\ldots =g(S_9)=1$, $g(S_{11})=\ldots =g(S_{19})=1$, 
$g(S_{21})=\ldots =g(S_{29})=1$, and $S_i\cdot S_i=-1$, for $i=1,\ldots,9,11,\ldots,19,21,\ldots,29$;
\item   $g(S_{10})=3$, $g(S_{20})=3$, $g(S_{30})=3$, and $S_j\cdot S_j=1$, $j=10,20,30$; 
\item $g(S_{31})=1$, $g(S_{32})=1$, $g(S_{33})=2$, and $S_{31}\cdot S_{31}=-1$, $S_{32}\cdot S_{32}=-1$,
 $S_{33}\cdot S_{33}=1$;
\item  $g(S_{34})=1$, $g(S_{35})=1$, $g(S_{36})=2$, and $S_{34}\cdot S_{34}=-1$, $S_{35}\cdot S_{35}=-1$,
$S_{36}\cdot S_{36}=1$.
\end{enumerate}
 The homology classes  $[S_j]$, $j=1, \ldots,36$, generate $H_2(X,\mathbb{Z})$. 
\end{theorem}

In the subsequent subsections we will construct such $X$. Our basic tools are Gompf symplectic sum, 
symplectic blow-up, elliptic and Lefschetz fibrations, and symplectic resolution of transverse intersections. 
We recall these tools following \cite{G-Stipsicz}.

\subsection{Symplectic resolution of transverse intersections}\label{subsec:resolution}
Let $X$ be a symplectic $4$-manifold and let $\Sigma_1$ and $\Sigma_2$ be embedded symplectic surfaces
intersecting transverseley and positively at a point $q\in X$. Then $\Sigma_1\cup\Sigma_2$ determines the 
homology class $[\Sigma_1]+[\Sigma_2]\in H_2(X,\mathbb{Z})$.  By Lemma $\ref{lem:symplectic orthogonal}$,
after slightly perturbing $\Sigma_1$ we can take Darboux coordinates 
$(z_1,z_2)$ in a $4$-ball neighbourhood $D$ of $q$, so that $\Sigma_1=\{z_1=0\}$ and $\Sigma_2=\{z_2=0\}$.
Then the union $\Sigma_1\cup\Sigma_2$ is described locally as 
 $$
  F=\{(z_1,z_2)\in D \,\,|\,\,z_1z_2=0,\,|z_1|^2+|z_2|^2\leq 1\}.
 $$
Cut out the pair $(D,F)$ and replace it with $(D,R)$, where $R\subset D$ is obtained by perturbing the subset
 $$
 R'=\{(z_1,z_2)\,\,|\,\,z_1z_2=\varepsilon,\,|z_1|^2+|z_2|^2\leq 1\},
 $$
for $\varepsilon>0$ sufficiently small, to achieve that $\partial F=\partial R\subset\partial D$. This construction 
replaces $\Sigma_1\cup\Sigma_2$ by a smooth symplectic surface of genus 
$g(\Sigma_1)+g(\Sigma_2)$, representing the homology class $[\Sigma_1]+[\Sigma_2]$. 
It does not change the ambient manifold.  We call this construction the resolution of the transverse intersection.

\subsection{Symplectic blow-up}\label{subsec:blowup}
Let $X$ be a symplectic $4$-manifold and $q\in X$. The symplectic blow-up of $X$ at $q$ is defined as follows.
Take the Darboux coordinates $(z_1,z_2)$ in a $4$-ball neighbourhood $D$ of $q$, and put the standard
complex structure $J$ on $D$. Consider 
 $$
 \tilde D=\{ ((z_1,z_2),[w_1,w_2])\in D\x \CP^1 \, |\, z_1w_2=z_2w_1\}.
 $$
Then there is a natural projection $q:\tilde D\to D$, such that $q:\tilde D -E\to D-\{(0,0)\}$ is a biholomorphism,
where $E=\{(0,0)\}\x \CP^1$, $q(E)=\{(0,0)\}$. We cut out $D$ from $X$ and replace it with $\tilde D$, obtaining
the manifold $\tilde X$. The symplectic form of $X$ and the natural symplectic form of $\tilde D$ (coming from
its K\"ahler structure) can be glued to give a symplectic structure for $\tilde X$. 
As a smooth $4$-manifold, $\tilde X=X\#\overline{\CP^2}$ and 
$E=\overline{\CP^1}\subset\overline{\CP^2}$ is called the exceptional sphere. 
Its homology class $[E]$ is denoted by $e\in H_2(X',\mathbb{Z})=
H_2(X,\mathbb{Z})\oplus H_2(\overline{\mathbb{C}P^2},\mathbb{Z})$ and satisfies $e\cdot e=-1$.

Now consider a symplectic surface $\Sigma\subset X$ and blow up a point $p\in \Sigma$. Then 
we can take coordinates $(z_1,z_2)$ such that $\Sigma=\{z_1=0\}$. The surface $\tilde\Sigma\subset \tilde X$
defined in $\tilde D$ by the equations $z_1=w_1=0$ is called the \emph{proper transform} of $\Sigma$. It is
symplectic, and $[\tilde\Sigma]=[\Sigma]-e$. Therefore $[\tilde\Sigma]^2=[\Sigma]^2-1$. Moreover, the
exceptional divisor $E$ is symplectic and intersects $\tilde\Sigma$ transversely. Actually, the symplectic resolution
of the intersection of $\tilde\Sigma\cup E$ is $\Sigma$.

If $\Sigma_1$ and $\Sigma_2$ are two symplectic surfaces in $X$ intersecting transversely and
positively at a point $p$, blowing-up at $p$ and taking the proper transforms,  
we get two disjoint symplectic surfaces $\tilde{\Sigma}_1 , \tilde{\Sigma}_2 \subset \tilde X$. 
This is proved by taking a Darboux chart such that $\Sigma_1=\{z_1=0\}$ and $\Sigma_2=\{z_2=0\}$, 
which is possible since $\Sigma_1,\Sigma_2$ intersect transversely and positively.

\subsection{Gompf symplectic sum}\label{subsec:sympl-sum}
The following construction is introduced in \cite{Gompf}. Let $M_1$ and $M_2$ be
closed symplectic $4$-manifolds, and $N_1\subset M_1$, $N_2\subset M_2$ symplectic surfaces of
the same genus and with $N_1^2=-N_2^2$. Fix a symplectomorphism $N_1\cong N_2$.
If $\nu_j$ is the normal bundle to $N_j$, then
there is a reversing-orientation bundle isomorphism $\psi: \nu_1\rightarrow \nu_2$. 
Identifying the normal bundles $\nu_i$ with the tubular neighbourhoods $\nu(N_j)$ of $N_j$ in $M_i$, one 
has a symplectomorphism  $\varphi: \nu(N_1) -N_1 \rightarrow \nu(N_2) -N_2$ by composing $\psi$ with 
the diffeomorphism $x\mapsto \frac{x}{||x||^2}$ that turns each punctured normal fiber inside out. 
The Gompf symplectic sum $M_1\#_NM_2$ is the manifold obtained from 
$(M_1-N_1) \sqcup (M_2-N_2)$ by gluing with $\varphi$ above. It is proved in \cite{Gompf} that 
this surgery yields a symplectic manifold, denoted $M=M_1\#_{N_1=N_2} M_2$.  
The Euler characteristic of the Gompf symplectic sum is given by 
$\chi(M)=\chi(M_1)+\chi(M_2)-2\chi(N)$, where $N=N_1=N_2$.

\begin{lemma} \label{lem:glue}
Suppose that $S_1\subset M_1$ and $S_2\subset M_2$ are symplectic surfaces intersecting
transversely and positively with $N_1,N_2$, respectively, such that $S_1\cdot N_1=S_2\cdot N_2=d$. Then
$S_1,S_2$ can be glued to a symplectic surface $S=S_1\# S_2 \subset M_1 \#_{N_1=N_2} M_2$
with self-intersection $S^2=S_1^2+S_2^2$ and genus $g(S)=g(S_1)+g(S_2)+d-1$.
\end{lemma}

\begin{proof}
When doing the Gompf symplectic sum of $M_1,M_2$ along $N_1,N_2$, we arrange
the symplectomorphism $N_1\cong N_2$ to take the intersection points
$S_1\cap N_1$ to the points $S_2\cap N_2$. Then 
we have to take tubular neighbourhoods of $N_j$ by using
the symplectic orthogonal to $T_pN_j$ at each $p\in N_j\cap S_j$. If $S_j$ and $N_j$
intersect orthogonally with respect to the symplectic form, then $S_1$ and $S_2$
glue nicely to give a symplectic surface $S$ in the Gompf connected sum.
We can arrange that the intersection becomes orthogonal after a small symplectic
isotopy around the intersection point, as done in Lemma \ref{lem:symplectic orthogonal}.
The claim about the self intersection and the genus are straightforward.
\end{proof}

\subsection{Elliptic fibrations}\label{subsec:elliptic}
We begin with some recollections on elliptic and Lefschetz fibrations from \cite{Gompf,G-Stipsicz}.
A complex surface $S$ is an elliptic fibration if there is a holomorphic map $f: S\rightarrow C$ 
to a complex curve $C$ such that for generic $t\in C$ the preimages $f^{-1}(t)$ are smooth elliptic curves. 
The elliptic fibration $E(1)$ is defined on $\CP^2$ blown-up at $9$ points as follows. Take two generic cubics in 
$\CP^2$ given by polynomials $p_0([x:y:z])=0$, $p_1([x:y:z])=0$. 
These cubics intersect in $9$ points $p_1,\ldots, p_9$. 
Consider the pencil of cubics $t_0p_0+t_1p_1$ parametrized by $[t_0:t_1]\in\CP^1$. For any point
$q\in\CP^2- \{p_1,\ldots,p_9\}$ there is only one cubic $t_0p_0+t_1p_1$ going through $q$. 
This defines a map 
 $$
 f:\CP^2 - \{p_1,\ldots, p_9\}\rightarrow \CP^1,\quad f(q)=[t_0:t_1].
 $$
Blowing up $\CP^2$ at $p_1,\ldots ,p_9$, we get a K\"ahler surface $E(1)=\CP^2 \# 9 \overline{\CP^2}$
and the map $f$ extends to a $f:E(1)\to \CP^1$, which is an elliptic fibration. 

We will use the notion of vanishing cycle. Let $X$ be a K\"ahler manifold. A Lefschetz fibration on $X$ is a 
holomorphic map $f: X\rightarrow\Sigma$, where $\Sigma$ is a complex curve such that each critical 
point of $f$ has a local (complex) coordinate chart on which $f(z_1,z_2)= z^2_1+z_2^2$. Hence
a regular fiber $F_t=f^{-1}(t)$ of the Lefschetz fibration is given locally by the equation 
$z_1^2+z_2^2=t$, and we can suppose $t>0$ multiplying $(z_1,z_2) \in F_t$ by some complex number. 

For $\epsilon>0$ real and positive, the intersection 
$F_t\cap\mathbb{R}^2\subset\mathbb{C}^2$ yields a circle $x_1^2+x_2^2=\epsilon$ (here $z_j=x_j+iy_j$). 
This circle bounds a disc $D_\epsilon$ in $X$ defined by $\{(z_1,z_2) \in F_t \cap\RR^2 \, | \, t\in [0,\epsilon]\}=
X\cap \RR^2\cap B_\epsilon(0)$, which is called the vanishing cycle of the critical point. This is
an embedded disc of self-intersection $-1$ and Lagrangian with respect to the symplectic structure of $X$.
We refer to \cite{G-Stipsicz} for the detailed exposition of the theory of elliptic and Lefschetz fibrations.

Let us summarize the properties of $E(1)$ which will be used later, from \cite{BPV,G-Stipsicz}.

\begin{proposition}\label{prop:9-blowup} The elliptic fibration $E(1)$ has the following properties.
\begin{enumerate}
\item $\pi_1(E(1))=\{e\}$, $\chi(E(1))=12$, $b_2(E(1))=10$.
\item Every exceptional sphere $E_i$ of the blow-up at a point $p_i$ is a section of the  elliptic 
fibration $f:E(1)\rightarrow\CP^1$, hence there are $9$ disjoint sections.
\item Let $h\in H_2(\CP^2,\mathbb{Z})$ be the homology class of the line $L\subset \CP^2$, and 
$e_i$ are homology classes of exceptional spheres $E_i$, then
$H_2(E(1),\ZZ)=\langle h,e_1,\ldots ,e_9\rangle$.
\item Let $F$ be a generic fiber of $f:E(1)\rightarrow\CP^1$. Then $\pi_1(E(1) -  F)=\{e\}$.
\item The homology class of $F$ is $[F]=3h-e_1-\ldots-e_9$. Hence a generic line $L\subset\CP^2$ 
intersects $F$ transversely in $3$ points.
\item For generic cubics, $f: E(1)\rightarrow\CP^1$ is a Lefschetz fibration.
\item In a fiber $F$ there are $12$ vanishing cycles; they come in two packets of six
$1$-cycles homologous to $a$ and six $1$-cycles homologous to $b$, where $\{a,b\}$ is a basis
of $H_1(F,\ZZ)$.
\end{enumerate}
\end{proposition}

We have the following result. 

\begin{lemma} \label{lem:pi1}
Let $X$ be a symplectic $4$-manifold with an embedded symplectic surface $T\subset X$ of self-intersection zero
and genus $1$. Then the Gompf connected sum $X'=X\#_{T=F}E(1)$ has fundamental group 
$\pi_1(X')=\pi_1(X)/H$, where $H$ is the normal subgroup generated by the image of $\pi_1(T)\to \pi_1(X)$.
\end{lemma}

\begin{proof}
By definition $X'=(X-\nu(T)) \cup_B (E(1)-\nu(F))$, where $B=\partial (X-\nu(T))=
\partial (E(1)-\nu(F)) \cong \TT^3$. Applying Seifert-Van Kampen theorem, $\pi_1(X')$ is isomorphic to the
amalgamated product $\pi_1(X-\nu(T)) \ast_{\pi_1(B)} \pi_1(E(1)-\nu(F))$. Since $\pi_1(E(1)-\nu(F))
=\{1\}$, this is isomorphic to the quotient of $\pi_1(X-\nu(T))$ by the image of $\pi_1(B)$. 
Using Seifert-Van Kampen theorem for $X=(X-\nu(T)) \cup_B \nu(T)$, $\pi_1(X)$ is isomorphic to 
$\pi_1(X-\nu(T)) \ast_{\pi_1(B)} \pi_1(\nu(T))$. Therefore the quotient of 
$\pi_1(X)$ by the image of $\pi_1(T)$ equals the quotient of $\pi_1(X-\nu(T))$ by the
image of $\pi_1(B)$. The result follows.
\end{proof}

\subsection{Making Lagrangian submanifolds symplectic}\label{subsec:lagrangian}
We will need a slight modification of  Lemma 1.6  in \cite{Gompf}.

\begin{lemma}\label{lemma:lagr-sympl} Let $(M,\omega)$ be a $4$-dimensional  compact symplectic manifold. 
Assume that  $[F_1],\ldots ,[F_k] \in H_2(M,\ZZ)$ are linearly independent homology classes represented by 
$k$ Lagrangian surfaces $F_1,\ldots F_k$ which intersect transversely and not three of them intersect in
a point. Then there is an arbitrarily small perturbation $\omega''$ of the symplectic form $\omega$ such that 
all $F_1,\ldots,F_k$ become symplectic.
\end{lemma}

\begin{proof} 
Since $[F_1],\ldots, [F_k]$ are linearly independent, there exists a closed $2$-form $\eta$
such that $\int_{F_i}\eta=1$, for all $i=1,\ldots, k$. Take symplectic (volume) forms $\omega_i$ 
on $F_i$ such that $\int_{F_i}\omega_i=1$. Then $\int_{F_i}(\omega_i-j^*_i\eta)=0$ so there are $1$-forms
$\alpha_i$ on $F_i$ such that $\omega_i-j^*_i\eta=d\alpha_i$.
 
We extend $\alpha_{i}$ to a tubular neighbourghoods $U_i$ of $F_i$ by pulling-back via a projection $p_i:U_i\to F_i$.
We arrange this projection to project any surface intersecting $F_i$ to a point. Then we extend $p_i^*\alpha_i$
to the whole of $M$ by multiplying with a cut-off function $\rho_i$ which is $0$ off a neighbourhood of $F_i$
and $1$ in a smaller neighbourhood. Set $\eta'=\eta+\sum_{j}d(\rho_j (p_j^*\alpha_{j}))$. 
Clearly, $d\eta'=d\eta=0$ and $j_{i}^*\eta'=\omega_{i}$, for all $i$. The form $\omega¡=\omega+\epsilon\eta'$ 
is symplectic for small $\epsilon>0$, and all $F_i$ are symplectic with respect to $\omega'$. 
\end{proof} 

\subsection{First step: a configuration of  tori in $\TT^4$}\label{subsec:tori}
Let $\mathbb{T}^4=\mathbb{R}^4/\mathbb{Z}^4$, with coordinates $x_1,\ldots,x_4$.
There are six embedded tori:
 \begin{align*}
 & T_{12}=\{(x_1,x_2,\alpha_3,\alpha_4)\}\subset\mathbb{T}^4,\,
 T_{34}=\{(\alpha_1,\alpha_2,x_3,x_4)\}\subset\mathbb{T}^4, \\
 &T_{23}=\{(\beta_1,x_2,x_3,\beta_4)\}\subset\mathbb{T}^4,\,
 T_{14}=\{(x_1,\beta_2,\beta_3,x_4)\}\subset\mathbb{T}^4,\\
 & T_{13}=\{(x_1,\gamma_2,x_3,\gamma_4)\}\subset\mathbb{T}^4,\,
 T_{24}=\{(\gamma_1,x_2,\gamma_3,x_4)\}\subset\mathbb{T}^4,
 \end{align*}
where $\alpha_i,\beta_i,\gamma_i$ are generic numbers which may be fixed when necessary. 
We get a configuration of six tori intersecting transversely in pairs $T_{12}\cap T_{34}$, 
$T_{23}\cap T_{14}$ and $T_{13}\cap T_{24}$, each pair intersects in a single point. 
The choice of the generic numbers ensures that one can have ``parallel'' disjoint copies 
$T_{ij}'$ of $T_{ij}$.
 
Consider a symplectic form 
 $$
 \omega=dx_1\wedge dx_2+dx_3\wedge dx_4+dx_2\wedge dx_3+ \delta \, dx_1\wedge dx_4
 +  dx_2\wedge dx_4 - \delta\, dx_1\wedge dx_3,
 $$ 
where $\delta>0$ is small. Note that all $T_{ij}$ are symplectic
with respect to $\omega$, where $T_{13}$ is given the reversed orientation.
From here we easily see that the following intersections are positive
$[T_{12}] \cdot [T_{34}]>0$ , $[T_{13}] \cdot [T_{24}]>0$ , $[T_{14}] \cdot [T_{23}] >0$.

Consider now the following specific collection of three disjoint $2$-tori, which are symplectic in $(\TT^4,\omega)$,
 \begin{align*}
  T_{12} &=\{(x_1,x_2,0,0)\}\subset\mathbb{T}^4, \\
 T_{13} &=\{(x_1,0,x_3,\frac12 )\}\subset \mathbb{T}^4,\\
 T_{14} &=\{(x_1,\frac12,\frac12,x_4)\}\subset\mathbb{T}^4
 \end{align*}
We shall do a Gompf connected sum along each of $T_{12}, T_{13}$ and $T_{14}$. For this, 
we cut out tubular neighbourhoods of $T_{12}, T_{13}$ and $T_{14}$ of some small radius $\varepsilon>0$. 
\begin{align*}
  Y &= \mathbb{T}^4  -  (\nu(T_{12}) \cup \nu(T_{13})\cup \nu(T_{14}))  \\ &=
 \{(x_1,x_2,x_3,x_4)\,\,|\,\,|(x_3,x_4)|\geq \varepsilon, \, |(x_2,x_4-\frac12)|\geq\varepsilon, \, 
|(x_2-\frac12,x_3-\frac12)|\geq\varepsilon\}. 
 \end{align*}
We shall denote $\bd_{1j}Y=\bd \nu(T_{1j})$, $j=2,3,4$, the three connected components
of the boundary $\bd Y$.

Let us describe a configuration of certain Lagrangian tori and cylinders in $Y$ to be used later.
\begin{align*}
 C_1 &=\{\left(x_1,- \delta (\frac12-2\varepsilon) (t-1), 0, \varepsilon+ (\frac12-2\varepsilon) t\right) ,\,t\in [0,1]\},\\
 C_2 &=\{\left(x_1, \frac12 +\delta (\frac12-2\varepsilon) (t-1), \varepsilon+ (\frac12-2\varepsilon) t,0 \right), \,t\in [0,1]\}, \\ 
 T_1 &=\{\left(\frac12 - \frac{\varepsilon}{\delta} (\sin \theta-\cos\theta), \varepsilon \cos \theta, x_3, \frac12+\varepsilon \sin\theta\right),\,\theta\in [0,2\pi]\}, \\
 T_2&=\{\left(\frac12 - \frac{\varepsilon}{\delta} (\sin \theta + \cos\theta), \frac12+\varepsilon \cos \theta, \frac12+\varepsilon \sin\theta,x_4\right),\,\theta\in [0,2\pi]\}.
 \end{align*}

 \begin{proposition}\label{prop:cylinders} If we choose $\delta$ and $\varepsilon$ small enough,
 the cylinders $C_1,C_2$ and the tori $T_1,T_2$ satisfy the following:
 \begin{enumerate}
 \item $C_1,C_2\subset Y$, $T_1\subset \partial_{13} Y$, $T_2 \subset \partial_{14} Y$, 
 \item $C_1\cap C_2=\emptyset$, $C_1\cap T_2=\emptyset$, $C_2\cap T_1=\emptyset$, $T_1\cap T_2=\emptyset$,
 \item $C_1$ and $T_1$ intersect transversely in one point, and the same holds for $C_2$ and $T_2$,
 \item $C_1,C_2,T_1,T_2$ are Lagrangian,
 \item $\partial C_1\subset \partial Y$ consists of two circles, one contained $\bd_{12} Y$ and another in $\bd_{13} Y$, 
 $\partial C_2\subset \partial Y$ consists of two circles, one contained $\bd_{12} Y$ and another in $\bd_{14} Y$.
 \end{enumerate}
 \end{proposition} 

 \begin{proof} 
The proof is obtained by a straightforward check up. It is easy to see that all of them are Lagrangian. For instance,
for $T_1$, the tangent space is generated by $-\frac{\varepsilon}{\delta}(\cos \theta+\sin \theta)\frac{\bd}{\bd x_1} 
- \varepsilon \sin\theta
\frac{\bd}{\bd x_2} + \varepsilon\cos\theta \frac{\bd}{\bd x_4}$ and $\frac{\bd}{\bd x_3}$, and
 $$
 \omega\left(-\frac{\varepsilon}{\delta} (\cos \theta+\sin \theta) \frac{\bd}{\bd x_1} - \varepsilon \sin\theta
\frac{\bd}{\bd x_2} + \varepsilon\cos\theta \frac{\bd}{\bd x_4},\frac{\bd}{\bd x_3}\right)=0.
 $$
The torus $T_1\subset \partial_{13}Y$ since its coordinates satisfy $|(x_2,x_4-\frac12)|=\varepsilon$. Analogously
$T_2 \subset \partial_{14} Y$.

Now, $\bd C_1=\{(x_1,0,0,\varepsilon)\}\sqcup \{(x_1,-\delta(\frac12-2\varepsilon),0,\frac12-\varepsilon)\}\subset
 \bd_{12}Y \sqcup \bd_{13}Y$, and clearly $C_1\subset Y$ since $x_3=0$ assures that it is well away from $\nu(T_{14})$.
The statement for $C_2$ is similar.

It is clear that $C_1\cap C_2=\emptyset$. For $T_1\cap T_2=\emptyset$, it follows since they are in different boundary components.
Also $ C_1\cap T_2=\emptyset$ since $T_2\subset \bd_{14}Y$ and $\bd C_1\subset \bd_{12}Y\sqcup \bd_{13}Y$. 
Similarly, $C_2\cap T_1=\emptyset$.
To compute $C_1\cap T_1$, looking at the fourth coordinate we see that they coincide only 
for $t=1, \theta= - \pi/2$ so there is only one point in $C_1\cap T_1$. 
In the same way, $C_2\cap T_2$ consists of one point (looking at the third coordinate). 
It is easy to check that these intersections are transverse.  
  \end{proof}

\subsection{Second step: The symplectic manifold $Z$}\label{subsec:second}
The normal bundles of $T_{1j}\subset \TT^4$ are trivial. Therefore we can take three 
copies of the elliptic surface $E(1)$, call them $E(1)_2$, $E(1)_3$ and $E(1)_4$, with
generic fibers $F_2,F_3,F_4$, respectively, and form the Gompf symplectic sum
 \begin{equation}\label{eqn:Z}
 Z=\mathbb{T}^4\#_{T_{12}=F_2} E(1)_2\#_{T_{13}=F_3} E(1)_3\#_{T_{14}=F_4} E(1)_4\, .
 \end{equation}
Using Lemma \ref{lem:pi1}, we have that $\pi_1(Z)$ is isomorphic to the quotient of
$\pi_1(\TT^4)$  by the images of  $\pi_1(T_{12}), \pi_1(T_{13}), \pi_1(T_{14})$, hence
$Z$ is simply-connected.
Using the formula for the Euler characteristic of the Gompf symplectic sum in 
Subsection \ref{subsec:sympl-sum}, one obtains 
 $$\chi(Z)=36,\,b_2(Z)=34.$$

Now we are going to construct $34$ symplectic surfaces in $Z$. This will be done in several steps.
First, let us focus on the first Gompf symplectic sum $\mathbb{T}^4\#_{T_{12}=F_2} E(1)_2$. Call
$E(1)=E(1)_2$, $F=F_{12}$, $T=T_{12}$.  
By Proposition \ref{prop:9-blowup} there are $9$ sections $E_1,\ldots ,E_9$  
of $E(1)$ which are spheres of self-intersection numbers $(-1)$ intersecting $F$ transversely at
one point. By Lemma \ref{lem:glue}, we can glue them to (disjoint parallel copies of) $T_{34}$,
to get $S_1=E_1\# T_{34},\ldots, S_9=E_9 \# T_{34}$, which are disjoint symplectic tori of self-intersection $-1$.
Now take a generic line $L\subset E(1)$ provided by Proposition \ref{prop:9-blowup}, which intersects
$F$ in three points (and does not intersect any of the exceptional spheres $E_i$). 
This is a symplectic sphere which can be glued, by Lemma \ref{lem:glue}, to three
parallel copies of $T_{34}$, to get a symplectic surface $S_{10}=L\# 3T_{34}$ of genus $3$ and self-intersection $1$,
which is moreover disjoint from all the previous ones.

When doing the second and third Gompf symplectic sums in (\ref{eqn:Z}), we construct similar collections
$S_{11},\ldots, S_{19}, S_{20}$ and $S_{21},\ldots, S_{29}, S_{30}$ of symplectic surfaces in $Z$, so that 
\begin{itemize}
\item $g(S_1)=\ldots =g(S_9)=1$, $g(S_{11})=\ldots =g(S_{19})=1$, $g(S_{21})=\ldots =g(S_{29})=1$,
$g(S_{10})= g(S_{20})= g(S_{30})=3$.
\item $S_k\cdot S_k=-1$, $1\leq k\leq 9$, $S_{10}\cdot S_{10}=1$, $S_{10+k}\cdot S_{10+k}=-1$, $1\leq k\leq 9$, 
$S_{20}\cdot S_{20}=1$,  
$S_{20+k}\cdot S_{20+k}=-1$, $1\leq k\leq 9$, $S_{30} \cdot S_{30}=1$.
\end{itemize}
All of them are disjoint since for constructing $S_{10+k}$, $k=1,\ldots, 10$, we glue with parallel copies of $T_{24}$,
and for constructing $S_{20+k}$, $k=1,\ldots, 10$, we glue with parallel copies of $T_{23}$.
We can arrange as many copies as we wish of $T_{34}, T_{24}, T_{23}$ which do not intersect.

The four remaining surfaces are constructed as follows. 
Consider the (Lagrangian) cylinders $C_1$,$C_2$ and tori $T_1,T_2$ from Proposition \ref{prop:cylinders}. 
Recall that they are contained in $Y$, so they are disjoint with the tori $T_{1j}$, $j=2,3,4$. Moreover, we can take 
collections of parallel copies of  $T_{34}, T_{24}, T_{23}$ which
do not intersect any of $C_1$,$C_2, T_1,T_2$. Therefore we can assume that $C_i$ and $T_i$ 
are disjoint from $S_1,\ldots,S_{30}$ in $Z$. 

We use the cylinder $C_i$ to construct Lagrangian spheres in $Z$ as follows. The boundary of $\bd C_1$ in 
$\bd_{12}Y$ is a circle $\gamma$. We arrange the identification $\bd( E(1)_2-\nu(F_2))=F_2 \x S^1 \cong \bd_{12} Y=
T_{12} \x S^1$ to match this circle with a vanishing cycle of the elliptic fibration $E(1)_2$ (see Subsection \ref{subsec:elliptic}).
Let $V$ be the vanishing disk in $E(1)_2$, which is a Lagrangian $(-1)$-disk. This can be glued to $C_1$ to obtain
a Lagrangian submanifold $V\cup_{\gamma} C_1$ of self-intersection $-1$. To make the gluing smooth, we
may need to change the gluing in the Gompf connected sum as follows: the gluing region is a neighbourhood of
$Y=F \x S^1$ of the form $F \x S^1 \x (-\epsilon,\epsilon)$, where the symplectic form is $\omega_F +  d\theta\wedge d t$, 
and the Lagrangian has tangent space at is spanned by $\gamma'$ and a vector $a\frac{\bd}{\bd\theta}+b\frac{\bd}{\bd t}$.
A diffeomorphism of the form $(\theta,s) \mapsto (\theta+g(s),s)$ can serve to arrange $a=0$, so that the Lagrangian
enters the gluing region in the radial direction and thus can be glued without corner.
Finally, gluing the other boundary component of $\bd C_1$ with a vanishing disk in $E(1)_{3}$, 
we get a Lagrangian $(-2)$-sphere $L_1$.
This intersects $T_1$ transversely at one point.

In a similar way we obtain another pair $L_2,T_2$ of a Lagrangian $(-2)$-sphere and Lagrangian torus of self-intersection $0$, both
intersecting transversely at one point. We can arrange that $L_1,L_2$ are disjoint, because by Proposition
\ref{prop:9-blowup} we can choose two different vanishing cycles (hence disjoint) in $E(1)_{2}$, to match
the two boundary components of $C_1,C_2$ in $\bd_{12}Y$, which are homologous cycles.

Looking at the intersection form, we see that the $34$ surfaces $S_1,\ldots, S_{30}$ and $L_1,L_2,T_1,T_2$ are independent
in homology, hence they span $H_2(Z,\QQ)$. Finally, we apply Lemma \ref{lemma:lagr-sympl} to change slightly the
symplectic form so that all these Lagrangian surfaces become symplectic. 
Moreover, the proof of Lemma \ref{lemma:lagr-sympl} shows that we can deform the symplectic 
form so that both pairs $(L_1,T_1)$ and $(L_2,T_2)$ intersect positively, so we assume this.

\subsection{Making all symplectic surfaces disjoint} \label{subsec:last-step}
To make the surfaces in $Z$ disjoint we have to do the following process with both pairs $L_1,T_1$ and $L_2,T_2$. 
Let $L,T$ a pair of a symplectic sphere and a symplectic torus with $L\cdot L=-2,L\cdot T=1, T\cdot T=0$. Take
a parallel copy of $T$, call it $T'$, displacing via the normal bundle. Resolve the intersection point 
$T'\cap L$ with the process of Subsection \ref{subsec:resolution} to get a torus $T''$ homologous to $T'+L$.
Hence $T''\cdot T'' = (T+L)^2=0$ and $T'' \cdot T= (T+L)\cdot T=1$. 
Therefore $T$ and $T''$ intersect  at one point, say $p$. Locally, the model around $p$ is determined by the equation $z\cdot w=0$, 
where $T=\{z=0\}$, and $T''=\{w=0\}$. Consider $T+T''$ and resolve 
the singularity producing a symplectic genus $2$ surface $\Sigma$. We move it to intersect $T$ and $T''$ 
in the same point $p$. Locally, it is the same as to write down the equation 
$(z-\varepsilon)\cdot (w-\varepsilon)=\varepsilon^2$. The equalities
  $$
   \Sigma\cdot T=(T+T'')\cdot T=1,\,\Sigma\cdot T''=(T+T'')\cdot T''=1,
   $$
show that $p$ is the only intersection point of the three surfaces $T,T'',\Sigma$, and that they intersect 
transversely. Moreover, $\Sigma^2=(T+T'')^2=2$. Blowing up at $p$ we get a symplectic 
manifold $\tilde Z=Z\# \overline{\mathbb{C}P^2}$, where the proper transforms  $\tilde T,\tilde T'',\tilde\Sigma$ 
are disjoint symplectic surfaces of genus
$1,1,2$ and self-intersection numbers $-1,-1,1$ (see Subsection \ref{subsec:blowup}). 
They generate the same $3$-dimensional space in homology, 
as $T,T''$ and the exceptional sphere $E\in H_2(\tilde Z,\ZZ)$. 

Using this method for both pairs $L_1,T_1$ and $L_2,T_2$, we end up with the symplectic manifold
$X=Z\#2\overline{\mathbb{C}P^2}$, with $b_2(X)=36$, and with $36$ disjoint symplectic surfaces $S_1,\ldots, S_{30}, 
\tilde T_1,\tilde T''_1,\tilde\Sigma_1, \tilde T_2,\tilde T''_2,\tilde\Sigma_2$. This forces that these $36$ surfaces generate 
the homology of $X$. The genus and self-intersections of the surfaces are those stated in Theorem \ref{thm:generators}.
This finishes the proof.

\begin{corollary} \label{cor:last}
Take a prime $p$, and $g_i=g(S_i)$ as given in Theorem \ref{thm:generators}. Then
there is a $5$-dimensional K-contact manifold $M$ with $H_1(M,\ZZ)=0$ and 
 $$
 H_2(M,\ZZ)= \ZZ^{35} \oplus \bigoplus_{i=1}^{36} (\ZZ/p^{i})^{2g_i}\, .
 $$
\end{corollary}

\begin{proof}
 Consider the symplectic manifold $(X,\omega)$ provided by Theorem \ref{thm:generators},
and let $S_i$, $1\leq i\leq 36$, be the collection of disjoint symplectic surfaces. Put coefficients
$m_i=p^{i}$ for $S_i$. Using Proposition \ref{prop:orb->symp}, we give $X$ the structure of
a symplectic orbifold with isotropy surfaces $S_i$ of multiplicities $m_i$. By Lemma \ref{lem:admits-aK}, 
$X$ admits an almost K\"ahler orbifold structure. 
Lemma \ref{lem:omeg}
implies that there exists a Seifert bundle $M\to X$ such that $c_1(M/X)=[\omega]$, and
by Theorem \ref{thm:k-contact}, $M$ admits a K-contact structure.

We compute the homology of $M$ using Theorem \ref{thm:Kollar}. As $X$ is simply connected,
$H_1(X,\ZZ)=0$. By Lemma \ref{lem:omeg}, we can arrange
that $c_1(M/\mu)\in H^2(X,\ZZ)$ is primitive. Now $k+1=b_2(X)=36$, so $H^2(X,\ZZ)=\ZZ^{36}$. The map
$H^2(X,\ZZ)\to H^2(S_i,\ZZ)$ sends $[S_j]$ to zero, $j\neq i$, since all $S_i$ are disjoint.
It sends $[S_i]$ to $S_i^2$, hence $H^2(X,\ZZ)\to H^2(S_i,\ZZ/p^{e_i})$ sends $[S_i]$ to $S_i^2 \pmod{p^{i}}$.
Given the self-intersection numbers in Theorem \ref{thm:generators}, this is $+1$ or $-1$. So 
 $$
  H^2(X,\ZZ) \to \sum H^2(S_i,\ZZ/p^{i} )
 $$
is surjective. Hence $H_1(M,\ZZ)=0$. The result follows. 
\end{proof}

 \begin{remark} \label{rem:suggestion}
 The manifold $M$ of Corollary \ref{cor:last} does not admit a regular K-contact structure.
This follows from Remark \ref{rem:e2} since $H_2(M,\ZZ)$ has torsion.
\end{remark}

\section{K\"ahler surfaces with many disjoint complex curves}\label{sec:6}

Now we want to find obstructions for the existence of Sasakian $5$-dimensional manifolds. In particular, we aim
to prove that the $5$-manifold constructed in the previous section, which admits a K-contact structure, cannot
admit a Sasakian structure.

The proof of Theorem \ref{thm:main} follows from Corollary \ref{cor:last} and the following: 

\begin{proposition}
Let $M$ be a $5$-dimensional manifold with $H_1(M,\ZZ)=0$ and 
 $$
 H_2(M,\ZZ)= \ZZ^{35} \oplus \bigoplus_{i=1}^{36} (\ZZ/p^{i})^{2g_i}\, .
 $$
where $g_i=g(S_i)$ are the numbers given in Theorem \ref{thm:generators}, and $p$ is a prime number.
Then $M$ does not admit a semi-regular Sasakian structure.
\end{proposition}

\begin{proof}
Let $M$ be a $5$-dimensional manifold with $H_1(M,\ZZ)=0$ which admits a Sasakian structure. Then it also admits
a quasi-regular Sasakian structure. This means that $M$ is a Seifert bundle over a K\"ahler orbifold $\pi:M\to X$,
by Theorem \ref{thm:otro}. By Corollary \ref{cor:Kollar}, $H_1(X,\ZZ)=0$, $H_2(X,\ZZ)=\ZZ^{36}$ and 
the ramification locus contains a collection of $36$ disjoint surfaces $D_i$ with $g(D_i) = g_i$. 

If the Sasakian structure is semi-regular, then $X$ is a smooth K\"ahler manifold. By Proposition 
\ref{prop:locus-orb-K}, the ramification locus consists of smooth K\"ahler curves which span the second homology (again, see Corollary \ref{cor:Kollar}). We see in Theorem \ref{thm:4g+5}
below that this is not possible.
\end{proof}

A smooth K\"ahler manifold with  disjoint complex curves spanning its second homology is a rare phenomenom. 
We have the following first result in this direction.

\begin{theorem} \label{thm:4g+5}
Let $S$ be a smooth K\"ahler surface with $H_1(S,\QQ)=0$ and 
containing $D_1,\ldots, D_b$, $b=b_2(S)$, smooth disjoint complex curves with $g(D_i)=g_i>0$,
and spanning $H_2(S,\QQ)$. Assume that: 
 \begin{itemize}
 \item at least two $g_i$ are bigger than $1$,
 \item $g=\max\{g_i\}\leq 3$.
 \end{itemize}
Then $b\leq 2g+3$.
\end{theorem}

\begin{proof}
First, it is clear that $[D_1],\ldots ,[D_b]$ are a basis of $H^2(S,\QQ)$. As these are classes of type
$(1,1)$, we have that $h^{1,1}=b$ and geometric genus $p_g=h^{2,0}=0$. The irregularity is
$q=h^{1,0}=0$ since $b_1=0$. Therefore Noether's formula \cite{BPV} says that 
 $$
 \frac{1}{12}(K_S^2+c_2(S)) = \chi(\cO_S)=1-q+p_g=1.
 $$
Note that $c_2(S)=\chi(S)=2+b$, since $b=b_2$ and $b_1=b_3=0$. Therefore
$K_S^2=10-b$, where $K_S$ is the canonical divisor of $S$.

By the Riemann-Hodge relations, the signature of $H^{1,1}(S)$ is $(1,b-1)$. Therefore, we
can suppose $D_1^2=m_1$, $D_i^2=-m_i$, $i=2,\ldots, b$, where all $m_i$ are positive integer numbers.
By the adjunction equality, we have
 $$
 K_S\cdot D_i+D_i^2 =2 g_i-2,
 $$
so $K_S\cdot D_i=2g_i-2-D_i^2$, and hence
 $$
K_S= \sum_{i=1}^b \frac{2g_i-2-D_i^2}{D_i^2} D_i
 $$
and
 $$
K_S^2= \sum_{i=1}^b  \frac{(2g_i-2-D_i^2)^2}{D_i^2}. 
 $$
For $i\geq 2$, we have 
 $$
  \frac{(2g_i-2-D_i^2)^2}{D_i^2} = - \frac{(2g_i-2+ m_i )}{m_i} (2g_i-2+m_i) \leq -(2g_i-2+m_i) \leq  -1,
$$
since $2g_i-2\geq 0$. Then
 $$
 10-b =K_S^2  \leq  \frac{(2g_1-2-D_1^2)^2}{D_1^2} -(b-1).
 $$
With the hypothesis that at least one $g_i$, $i\geq 2$, satisfies that $g_i>1$, we have an strict inequality. So
 $$
 \frac{(2g_1-2-m_1)^2}{m_1} \geq 10.
 $$ 
This is rewritten as $m_1^2-(4g_1+6)m_1+4(g_1-1)^2 \geq 0$. Hence
  $$
  m_1\geq 2g_1+3 +\sqrt{20g_1+5} \qquad \text{or} \qquad
  m_1\leq 2g_1+3 -\sqrt{20g_1+5}.
 $$
For $g_1\leq 3$, we have that the second inequality is impossible (since $m_1\geq 1$). Hence
$m_1\geq 2g_1+3$.

Now we have that there is a curve $D_1$ of genus $g_1$ with self-intersection $D_1^2\geq 2g_1-1$. Take
the line bundle $L=\cO(D_1)$. This has $m_1=\deg( L|_{D_1})\geq 2g_1-1$, so $L|_{D_1}$ is very ample.
In particular, there is a section $s\in H^0(L|_{D_1})$ vanishing exactly at $m_1$ distinct points $Z\subset D_1$. The 
long exact sequence in cohomology associated to $0\to \cO \to L\to L|_{D_1}\to 0$, together with the fact that 
$H^1(\cO)=H^{0,1}(S)=0$, gives an exact sequence
 $$
  0 \to \CC \to H^0(L) \to H^0(L|_{D_1}) \to 0,
 $$
Take the preimage of the section $s \in H^ 0(L|_{D_1})$. This is a $2$-dimensional subspace of
$H^0(L)$. It gives a Lefschetz pencil $\PP^1 \subset \PP(H^0(L))$ of sections whose
zero sets are curves going through $Z$. Blow-up $Z$ to get a smooth complex surface $\tilde S$ and a Lefschetz fibration
 $$
 \pi:\tilde S  \too \PP^1
$$
with the proper transform of $D_1$, say $C_1=\tilde D_1$ as one smooth fiber of genus $g_1$. The other
$D_i$, $2\leq i\leq b$, are not touched by the blow-up loci, so we do not change their name.

Now let $E_j$, $j=1,\ldots, m_i$, be the exceptional divisors of the blow-up map $\tilde S\to S$. These are sections
of $\pi$. Note that $C_1, E_1,\ldots, E_{m_1}, D_2,\ldots, D_b$ are a basis of $H_2(\tilde S,\QQ)$.
Since $D_j\cdot E_1 =0$, we have that $D_j$ is contained in a fiber for all $i=2,\ldots, b$. Note that
it follows that $g_i\leq g_1$, for $i\geq 2$, since the genus of a component of a singular fiber cannot be
bigger than the genus of the generic fiber. So $g=g_1$.

Let us see that $\pi$ is a relatively minimal fibration. This means that there are no $(-1)$-rational curves contained
in a fiber. Suppose that $B$ is such a curve. If $B$ intersects a section say $E_1$, then $B+E_1$ is a rational nodal
curve of self-intersection zero. This implies that there is a linear system of rational curves of self-intersection zero
and hence $\tilde S$ is ruled. 
Now note that there cannot be more than $g$ curves of genus $\geq 1$ in a fiber (since $g$ is the
genus of a generic fiber). If $b_2\leq g+1$ then automatically $b_2\leq 2g+3$. Otherwise 
$b_2\geq g+2$ and then there must be a curve $D_3$ not lying in the fiber of $B$. Then 
$(B+E_1)\cdot D_3 =0$, and $D_3$ has positive genus. 
This is not possible (a
curve of positive genus survives in a minimal model of $\tilde S$, hence it should be intersected by the ruling). 

Suppose that $B$ does not intersect any section. Then $B$ is contained in a fiber. If $B$ does not intersect any
$D_j$ then it is homologically trivial. Suppose it intersect some $D_k$ in some fiber $F$. Let $F_1,\ldots, F_k$ be
the irreducible components of $F$. By \cite[(III.8.2)]{BPV}, the span of $\la F_1,\ldots, F_k\ra$ 
has dimension $k$, and subject to the only relation $C_1=F=\sum a_i F_i$, for some $a_i$. Removing
the components that do intersect the exceptional divisors, the rest of the components, together with the $D_i$
and the $E_j$, should be independent. Therefore there cannot be more components not intersecting the $E_j$ than
those provided by the $D_k$ in the fiber, hence such $B$ does not appear.

Now, as in \cite{X} and \cite{Chen}, write 
 \begin{align*}
 K_{\tilde S/\PP^1}^2 &= K_{\tilde S}^2-8(g-1)(-1)=10-b-m_1+8g-8, \\
 \chi_\pi &= \chi(\cO_{\tilde S})-(g-1)(-1)=1+g-1=g, \\
 \lambda_\pi &=K_{\tilde S/\PP^1}^2  / \chi_\pi= (2-b-m_1+8g)/g.
 \end{align*}
By \cite{X}, for any relatively minimal fibration of genus $g\geq 2$, we have $4-4/g \leq \lambda_\pi\leq 12$.
The first inequality implies that $4g-4 \leq 2-b-m_1+8g \leq 2-b-(2g+3)+8g$ hence $b\leq 2g+3$.
\end{proof}

\begin{remark}
The proof of Theorem \ref{thm:4g+5} also works when we have all complex curves spanning the second homology of genus $g_i=1$. We only have to 
note that automatically $m_1 \geq 1$, and this is enough to construct a Lefschetz fibration.
\end{remark}

To extend the arguments of this paper to quasi-regular Sasakian manifolds (and hence to all Sasakian manifolds), we
need a version of Theorem \ref{thm:4g+5} that covers the case that $S$ is a cyclic K\"ahler orbifold. The argument should 
run as follows: desingularize each orbifold point (this is a Hirzebruch-Jung desingularisation \cite{BPV}), creating a
tree of rational curves of negative self-intersection, and bound $K_{\tilde S}^2$ for the desingularisation $\tilde S \to S$.
The authors have only managed to make this argument work for the case where all complex curves are of genus $g_i=1$.
Unfortunately, we have not been able to construct a symplectic manifold $X$ with $H_1(X,\ZZ)=0$ and $b=b_2(X)$
disjoint symplectic tori in $X$ spanning $H_2(X,\QQ)$.

\end{document}